\documentclass[11pt,leqno,draft]{article}
\usepackage{amsfonts}
\usepackage{amsmath, amsthm, amsfonts, amssymb, color,dsfont}
\usepackage{mathrsfs}
\usepackage{color}
\newtheorem{thm}{Theorem}[section]

\newtheorem{lem}[thm]{Lemma}

\newtheorem{exa}[thm]{Example}
\newtheorem{rem}[thm]{Remark}
\theoremstyle{definition}

\newcommand{\scr}[1]{\mathscr #1}
\definecolor{wco}{rgb}{0.5,0.2,0.3}

\numberwithin{equation}{section} \theoremstyle{remark}

\newcommand{\ua}{\uparrow}

\title{{\bf   Convergence rate of EM algorithm for SDEs under integrability condition } \footnote{Supported in
 part by  NNSFC (11801406).} }
\author{
{\bf  Jianhai Bao$^{b)}$,  Xing Huang$^{a)}$,  Shao-Qin Zhang$^{c)}$}\\
\footnotesize{$^{a)}$Center for Applied Mathematics, Tianjin
University, Tianjin 300072, China}\\
\footnotesize{  xinghuag@tju.edu.cn}\\
\footnotesize{$^{b)}$Department of Mathematics, Swansea University,
Singleton Park, SA2 8PP, UK}\\
\footnotesize{Jianhai.Bao@Swansea.ac.uk}\\
\footnotesize{$^{c)}$School of Statistics and Mathematics, Central University of Finance and Economics, Beijing 100081, China}\\
\footnotesize{zhangsq@cufe.edu.cn}}
\begin{document}
\allowdisplaybreaks
\def\R{\mathbb R}  \def\ff{\frac} \def\ss{\sqrt} \def\B{\mathbf
B}
\def\N{\mathbb N} \def\kk{\kappa} \def\m{{\bf m}}
\def\ee{\varepsilon}\def\ddd{D^*}
\def\dd{\delta} \def\DD{\Delta} \def\vv{\varepsilon} \def\rr{\rho}
\def\<{\langle} \def\>{\rangle} \def\GG{\Gamma} \def\gg{\gamma}
  \def\nn{\nabla} \def\pp{\partial} \def\E{\mathbb E}
\def\d{\text{\rm{d}}} \def\bb{\beta} \def\aa{\alpha} \def\D{\scr D}
  \def\si{\sigma} \def\ess{\text{\rm{ess}}}
\def\beg{\begin} \def\beq{\begin{equation}}  \def\F{\scr F}
\def\Ric{\text{\rm{Ric}}} \def\Hess{\text{\rm{Hess}}}
\def\e{\text{\rm{e}}} \def\ua{\underline a} \def\OO{\Omega}  \def\oo{\omega}
 \def\tt{\tilde} \def\Ric{\text{\rm{Ric}}}
\def\cut{\text{\rm{cut}}} \def\P{\mathbb P} \def\ifn{I_n(f^{\bigotimes n})}
\def\C{\scr C}   \def\G{\scr G}   \def\aaa{\mathbf{r}}     \def\r{r}
\def\gap{\text{\rm{gap}}} \def\prr{\pi_{{\bf m},\varrho}}  \def\r{\mathbf r}
\def\Z{\mathbb Z} \def\vrr{\varrho} \def\ll{\lambda}
\def\L{\scr L}\def\Tt{\tt} \def\TT{\tt}\def\II{\mathbb I}
\def\i{{\rm in}}\def\Sect{{\rm Sect}}  \def\H{\mathbb H}
\def\M{\scr M}\def\Q{\mathbb Q} \def\texto{\text{o}} \def\LL{\Lambda}
\def\Rank{{\rm Rank}} \def\B{\scr B} \def\i{{\rm i}} \def\HR{\hat{\R}^d}
\def\to{\rightarrow}\def\l{\ell}\def\iint{\int}
\def\EE{\scr E}\def\no{\nonumber}
\def\A{\scr A}\def\V{\mathbb V}\def\osc{{\rm osc}}
\def\BB{\scr B}\def\Ent{{\rm Ent}}
\def\U{\scr U}\def\8{\infty} \def\si{\sigma}

\renewcommand{\bar}{\overline}
\renewcommand{\tilde}{\widetilde}
\maketitle

\begin{abstract}
In this paper, by employing  Gaussian type estimate of heat kernel,
we establish Krylov's estimate and Khasminskill's estimate for EM
algorithm. As applications, by taking Zvonkin's transformation into
account, we investigate convergence rate of EM algorithm for a class
of multidimensional  SDEs under integrability conditions, where the
drifts need not to be  piecewise Lipschitz and are much more
singular in some sense.

\end{abstract} \noindent
 \vskip 2cm

\section{Introduction and Main Results}
Strong/weak convergence of numerical schemes for stochastic
differential equations (SDEs for short) with regular coefficients
have been investigated considerably; see monographs e.g.
\cite{KP,KPS}. As we know, (forward) Euler-Maruyama (EM for
abbreviation) is the simplest algorithm to simulate SDEs whose
coefficients are of linear growth. Whereas,  EM scheme is invalid as
long as the coefficients of SDEs involved are of nonlinear growth;
see e.g. \cite{HJK,JMY} for some illustrative counterexamples.
Whence the other variants of EM scheme were designed  to deal with
SDEs with  non-globally Lipschitz condition; see e.g. \cite{HMS,HMY}
for backward EM scheme, \cite{DKS,HJP,Sa} as for tamed EM algorithm,
and \cite{GMY,Mao} concerning truncated EM method, to name a few.
Nowadays, convergence analysis of numerical algorithms for SDEs with
irregular coefficients also receives much attention; see e.g.
\cite{GR}  for SDEs with H\"older continuous diffusions via
Yamada-Watanabe approximation approach, \cite{Yan} for SDEs whose
drift terms are H\"older continuous with the aid of Meyer-Tanaka
formula and estimates on local times, and \cite{BHY,PT} for SDEs
whose drifts enjoy H\"older(-Dini) continuity by the regularity of
backward Kolmogrov equations. In the past few years, numerical
approximations of SDEs with discontinuous drifts have also gained  a
lot of interest; see, for instance,
  \cite{GLN,HK,LS,LS2,LS3,MY,NSS}. Up to now, most of the existing
  literatures above on strong approximations of SDEs with discontinuous
  drift coefficients are implemented under the additional assumption that the
  drift term is {\it piecewise Lipschitz continuous}.

Since the pioneer work of Zvonkin \cite{AZ}, the wellposedness for
SDEs under integrability conditions has been developed greatly  in
different manners; see e.g. \cite{FGP,GM,KR,XZ,Z,Z2} for SDEs driven
by Bronian motions or jump processes, and e.g. \cite{HW18,RZ} for
McKean-Vlasov (or distribution-dependent or mean-field) SDEs. So
far, there also exist a few of literatures upon numerical
simulations of SDEs under the integrability condition. In
particular, \cite{NT} is concerned with strong convergence rate of
EM scheme for SDEs with irregular coefficients, where the one-side
Lipschitz condition is imposed on the drift term. Subsequently,
the one-side Lipschitz condition put in \cite{NT} was dropped in
\cite{NT2} whereas the {\it $1$-dimensional} SDEs are barely
concerned. At this point, our goal in this paper has been evident.
More precisely,
 motivated by the previous literatures, in this paper we aim to
investigate convergence rate of EM for several class of {\it
multidimensional} SDEs under {\it integrability condition}, which
nevertheless allows that the drift terms need not to be piecewise
Lipschitz continuity imposed in e.g.
\cite{GLN,HK,LS,LS2,LS3,MY,NSS}.

Now we  consider the following SDE \beq\label{1.1} \d X_{t}=b(X_t)\d
t+\sigma(X_t)\d W_{t}, ~~~t\ge0,\ \ X_{0}=x,
\end{equation}
where $b:\R^d\to\R^d$, $\si:\R^d\to\R^d\otimes\R^m$, and
$(W_t)_{t\ge0}$ is an $m$-dimensional Brownian motion on some
filtered probability space $(\OO,\F,(\F)_{t\ge0},\P)$. For the drift
$b$ and the diffusion $\si$, we assume that \beg{enumerate}
\item[{\bf (A1)}] $\|b\|_{\infty}:=\sup_{x\in\R^d}|b(x)|<\infty$ and there is $p>\ff{d}{2}$ such
that $|b|^2\in L^p$, the usual  $L^p$ space on $\R^d$;
\item[{\bf (A2)}]  There exist a constant $\aa>0$ and a locally integrable  function $\phi\in C(\R_+;\R_+)$ such
that
\begin{equation*}
\frac{1}{ s^{\frac{d}{2}}}\int_{\mathbb{R}^d}|b(x+y)-b(x+z)|^2
\e^{-\frac{1}{s}|x|^2}\d x\leq
\phi(s)|y-z|^\aa,~~~~y,z\in\R^d,~~s>0;
\end{equation*}
\item[{\bf (A3)}] There exist constants $\breve{\ll}_0,\hat\ll_0,L_0>0$ such
that
\begin{equation}\label{F0}
\breve{\ll}_0|\xi|^2\le\<(\si\si^*)(x)\xi,\xi\>\le\hat\ll_0|\xi|^2,~~~x,\xi\in\R^d
\end{equation}
and
\begin{equation}\label{F00}
\|\si(x)-\si(y)\|_{\rm HS}\le L_0|x-y|,~~~x,y\in\R^d,
\end{equation}
where $\si^*$ means the transpose of $\si$ and $\|\cdot\|_{\rm HS}$
stands for the Hilbert-Schmidt norm.
\end{enumerate}

Under ({\bf A1}) and ({\bf A3}), \eqref{1.1} has a unique strong
solution $(X_t)_{t\ge0}$; see, for instance, \cite[Lemma 3.1]{HW18}.
Moreover, ({\bf A2}) imposed is to reveal the convergence rate of EM
scheme corresponding to \eqref{1.1}, which is defined as below: for
any $\dd\in(0,1), $
\begin{equation}\label{E1}
\d X^{(\dd)}_t= b( X^{(\dd)}_{t_\dd})\d t+
\sigma(X^{(\dd)}_{t_\dd})\d W_t, \ \ t\ge0,~~~X^{(\dd)}_0=X_0
\end{equation}
with $t_\dd:=\lfloor t/\dd\rfloor\dd$, where  $\lfloor t/\dd\rfloor$
denotes the integer part of $t/\dd$.  We emphasize that
$(X^{(\dd)}_{k\dd})_{k\ge0}$ is a homogeneous Markov process; see
e.g. \cite[Theorem 6.14]{MY06}. For $t\ge s$ and $x\in\R^d$, denote
$p^{(\dd)}(s,t,x,\cdot)$ by the transition density of $ X^{(\dd)}_t
$ with the starting point $X_s^{(\dd)}=x$.

Our first main result in this paper is stated as follows.

\beg{thm}\label{th1} Assume {\bf (A1)}-{\bf (A3)}. Then, for
$\bb\in(0,2)$ and $q>2$, there exist constants $C_1,C_2>0$ such that
\begin{equation}\label{W1}
 \E\Big(\sup_{0\le t\le T}|X_t-X_t^{(\dd)}|^\bb\Big)  \le C_1
\e^{C_2(1+\||b|^2\|_{ L^{p}}^{\gg_0} )}
\dd^{\ff{\bb}{2}(1\wedge\ff{\aa}{2})},
\end{equation}
where $\gg_0:=\ff{1}{1-1/q-d/2p}$.
\end{thm}

Compared with \cite{NT}, in Theorem \ref{th1} we get rid of the
one-side Lipschitz condition. On the other hand,  \cite{NT2} is
extended to the multidimensional setup. We point out that an
$\mathcal {A}$ approximation is given in advance in \cite{NT,NT2} to
approximate the drift term. So, with contrast to the assumption put
in \cite{NT,NT2},  the assumption ({\bf A2}) is much more explicit.
Moreover, by a close inspection of the argument of Lemma \ref{lem1},
the assumption ({\bf A2}) can be replaced by the other alternatives.
For instance, ({\bf A2}) may be taken the place of ({\bf A2'})
below.
\begin{enumerate}
\item[{\bf (A2')}] There exist  constants $\bb,\theta>0$ such that
\begin{equation}\label{B2}
\ff{1}{(rs)^{d/2}}\sup_{z\in\R^d}\int_{\R^d\times\R^d}|b(x)-b(y)|^2\e^{-\ff{|x-z|^2}{s}}\e^{-\ff{|y-x|^2}{r}}\d
y\d x\le C  r^\theta s^{\bb-1},~~~s,r>0
\end{equation}
for some constant $C>0.$
\end{enumerate}

The drift $b$ satisfying \eqref{B2} is said to the Gaussian-Besov
class with the index $(\bb,\theta)$, denoted by
$GB^2_{\bb,\theta}(\R^d)$. Remark that functions with the same order
of continuity may enjoy different type continuity; see, for
instance, $f(x)=|x|$ with $(1/2,1)$, and
$f(x)=\mathds{1}_{[c,d]}(x),c,d\in\R,$ with $(1/2,1/2).$ We refer
to Example \ref{exa1} below  for the drift $b\in
GB^2_{\bb,\theta}(\R^d)$. For $\theta\in(0,1)$ and $p\ge1,$ let
$W^{\theta,p}(\R^d)$ be the fractional order Sobolev space on
$\R^d$. Nevertheless, $W^{\theta,p}(\R^d)\not\subseteq
GB^2_{1-\ff{d}{p},\theta}(\R^d), \theta>0,p\in[2,\8)\cap(d,\8);$ see
Example

In Theorem \ref{th1}, the integrable condition (i.e., $|b|^2\in
L^p$) seems to be a little bit restrictive, which rules out some
typical examples, e.g.,  $b(x)={\bf1}_{[0,\8)}(x)$. In the sequel,
by implementing  a truncation argument,  the integrable condition
can indeed be dropped. In such setup (i.e., without integrable
condition), we can still derive the convergence rate of the EM
algorithm, which is presented as below.

\begin{thm}\label{th2}
Assume   {\bf (A1)}-{\bf (A3)}  without $|b|^2\in L^p$. Then, for
$\bb\in(0,2)$, and $p,q>2$ with $\ff{d}{p}+\ff{1}{q}<1$, there exist
constants $C_1,C_2>0$ such that
\begin{equation}\label{S4}
\E\Big(\sup_{0\le t\le T}|X_{t}-X^{(\dd)}_t|^\bb\Big)\leq
C_1\Big\{\e^{C_2(- \ff{\bb}{2}(1\wedge\ff{\aa}{2})\log\dd
)^{\ff{d\gg_0}{2p}}} +1\Big\}\dd^{\ff{\bb}{2}(1\wedge\ff{\aa}{2})}.
\end{equation}
\end{thm}

We remark that the right hand side of \eqref{S4} approaches zero
since \begin{equation*} \lim_{\dd0\8} \e^{C_2(-
\ff{\bb}{2}(1\wedge\ff{\aa}{2})\log\dd )^{\ff{d\gg_0}{2p}}}
\dd^{\ff{\bb}{2}(1\wedge\ff{\aa}{2})} =0.
\end{equation*}
due to the fact that
$\lim_{x\rightarrow\8}\ff{\e^{C_2x^{\ff{d\gg_0}{2p}}}}{\e^x} =0$
whenever $\ff{d}{p}+\ff{1}{q}<1$.

 The remainder of this paper is organized as follows. In Section
\ref{sec2}, by employing Zvonkin's transform and establishing
Krylov's estimate and Khaminskill's estimate for EM algorithm, which
is based on Gaussian type estimate of heat kernel,  we complete the
proof of Theorem \ref{th1}; In Section \ref{sec3}, we aim to finish
the proof of Theorem \ref{th2} by adopting a truncation argument; In
Section \ref{sec4} we provide some illustrative examples to
demonstrate our theory established; In the Appendix part,  we reveal
explicit upper bounds of the coefficients associated with Gaussian
type heat kernel of the exact solution and the   EM scheme.

\section{Proof of Theorem \ref{th1}}\label{sec2}
Before  finishing the proof of Theorem \ref{th1}, we prepare several
auxiliary lemmas. Set
\begin{equation}\label{C0}
\begin{split}
\Lambda_1:&=2\Big\{\ff{\|b\|_\8}{ \ss{\breve{\ll}_0}}+2\ss{d
}L_0(\hat\ll_0/\breve{\ll}_0)^{2} +  d^{\ff{d}{2}+1}d!
(\hat\ll_0/\breve{\ll}_0)^{d} L_0
\Big\}\e^{\ff{\|b\|_\8^2T}{\hat\ll_0}}\\
%
&\quad\vee\bigg\{   2\ss{\hat\ll_0} \|b\|_\8+
(\|b\|_\8^2+2\hat\ll_0L_0\ss d)(\ss d+2) +
2^{m+11}\breve{\ll}_0^{-1}(L_0+2\|b\|_\8)\\
&\quad\times\Big((\|b\|_\8^3+(d\hat\ll_0)^{\ff{3}{2}})+\breve{\ll}_0^{\ff{1}{2}}(\|b\|_\8^2
 +d\hat\ll_0)\Big) \bigg\}\ff{2^{\ff{d+1}{2}}}{\breve{\ll}_0}\e^{\ff{(\|b\|_\8
 +\|b\|_\8^2)T}{\hat\ll_0}},
\end{split}
\end{equation}
and
\begin{equation}\label{P1}
\Lambda_2:=\e^{\ff{\|b\|_\8T}{2\hat\ll_0}}\sum_{i=0}^{\8}\ff{\Big(\Lambda_1\ss{\pi
T} ((1+24d)\hat\ll_0/\breve{\ll}_0)^d\Big)^i}{\GG(1+\ff{i}{2})},
\end{equation}
where $\GG(\cdot)$ denotes the Gamma function.   Due to Stirling's
formula: $\GG(z+1)\sim\ss{2\pi z}(z/\e)^z$, we have  $\Lambda_2<\8$.

The lemma below provides an explicit upper bound of the transition
kernel for $(X_t^{(\dd)})_{t\ge0}$.

\begin{lem}\label{lem0}
 Under $({\bf A1})$ and $({\bf A3})$,
\begin{equation}\label{A03}
p^{(\dd)}(j\dd,t,x,y)\le \ff{\LL_3\e^{-\ff{|
y-x|^2}{\kk_0(t-j\dd)}}}{(2
\pi\breve{\ll}_0(t-j\dd))^{d/2}},~~~x,y\in\R^d,~~t>
j\dd,~~\dd\in(0,1),
\end{equation}
where
\begin{equation}\label{P2}
\kk_0:=4(1+24d)\hat\ll_0,~~~~~\Lambda_3:=\Lambda_2\e^{\ff{\|b\|_\8^2}{2\hat\ll_0}}\Big(\ff{\kk_0}{2\breve{\ll}_0}\Big)^{d/2}.
\end{equation}

\end{lem}
\begin{proof}
For fixed $t>0$, there is an integer $k\ge0$ such that
$[k\dd,(k+1)\dd).$ By a direct  calculation, it follows from
\eqref{F0} and \eqref{F00} that
\begin{equation}\label{A02}
\begin{split}
p^{(\dd)}(k\dd,t,x,y)
&\le\ff{\e^{-\ff{|y-x-b(x)(t-k\dd)|^2}{2\hat\ll_0(t-k\dd)}}}{(2\pi\breve{\ll}_0)(t-k\dd)^{d/2}}\le\e^{\ff{\|b\|_\8^2}{2\hat\ll_0}}\ff{\e^{-\ff{|y-x|^2}
{4\hat\ll_0(t-k\dd)}}}{(2\pi\breve{\ll}_0(t-k\dd))^{d/2}},
\end{split}
\end{equation}
where in the second inequality we  used the basic inequality:
$|a-b|^2\ge\ff{1}{2}|a|^2-|b|^2, a,b\in\R^d.$ Next, by invoking
Lemma \ref{THA} below, one has
\begin{equation}\label{A02}
p^{(\dd)}(j\dd,j'\dd,x,x')\le\ff{\Lambda_2\e^{-\ff{|
x'-x|^2}{\kk_0(j'\dd-j\dd)}}}{(2
\pi\breve{\ll}_0(j'\dd-j\dd))^{d/2}},~~~j'>j,~x,x'\in\R^d,
\end{equation}
where $\LL_2,\kk_0$ were given in \eqref{P1} and \eqref{P2},
respectively. Subsequently, \eqref{A03} follows immediately by
taking advantage of the facts that
\begin{equation*}
\begin{split}
p^{(\dd)}(j\dd,t,x,y)=\int_{\R^d}p^{(\dd)}(j\dd,\lfloor
t/\dd\rfloor\dd,x,u)p^{(\dd)}(\lfloor t/\dd\rfloor\dd,t,u,y)\d u,
\end{split}
\end{equation*}
which is due to the Chapman-Kolmogrov equation, and
\begin{equation*}
\int_{\R^d}\ff{\e^{-\ff{| u-x|^2}{\kk_0(k\dd-j\dd)}}}{(2
\pi\breve{\ll}_0(k\dd-j\dd))^{d/2}}\ff{\e^{-\ff{|y-u|^2}
{4\hat\ll_0(t-k\dd)}}}{(2\pi\breve{\ll}_0(t-k\dd))^{d/2}}\d
u\le\Big(\ff{\kk_0}{2\breve{\ll}_0}\Big)^{d/2} \ff{\e^{-\ff{|
y-x|^2}{\kk_0(t-j\dd)}}}{(2 \pi\breve{\ll}_0(t-j\dd))^{d/2}},~~k>j.
\end{equation*}
\end{proof}

\beg{lem}\label{lem1} Under  {\bf (A1)}-{\bf (A3)},  for any $T>0$,
there exists a constant  $C>0 $ such that
\begin{equation} \label{b-b}
\int_0^T\mathbb{E}|b(X^{(\dd)}_t)-b(X^{(\dd)}_{t_\dd})|^2\d t\leq C
\dd^{1\wedge\ff{\aa}{2}}.
 \end{equation}
\end{lem}

\begin{proof}
Observe  that
\begin{equation*}
\begin{split}
\int_0^T\mathbb{E}|b(X^{(\dd)}_t)-b(X^{(\dd)}_{t_\dd})|^2\d
t&=\int_0^\dd\mathbb{E}|b(X^{(\dd)}_t)-b(X^{(\dd)}_0)|^2\d
t\\
&\quad+\sum_{k=1}^{\lfloor
T/\dd\rfloor}\int_{k\dd}^{T\wedge(k+1)\dd}\mathbb{E}|b(X^{(\dd)}_t)-b(X^{(\dd)}_{k\dd})|^2\d
t.
\end{split}
\end{equation*}
By $\|b\|_\8<\8$ due to ({\bf A1}), it follows that
\begin{equation}\label{F2}
\int_0^\dd\mathbb{E}|b(X^{(\dd)}_t)-b(X^{(\dd)}_0)|^2\d t\le
4\|b\|_\8^2\dd.
\end{equation}
For  $t\in[k\dd,(k+1)\dd)$,  by taking the  mutual independence
between $X^{(\dd)}_{k\dd}$ and $W_t-W_{k\dd}$ into account and
employing
  Lemma \ref{lem0}, we derive  that
\begin{equation}\label{D8}
\begin{split}
&\mathbb{E}|b(X^{(\dd)}_t)-b(X^{(\dd)}_{k\dd})|^2\\&
=\E|b(X^{(\dd)}_{k\dd}+b(X^{(\dd)}_{k\dd})(t-k\dd)+\si(X^{(\dd)}_{k\dd})(W_t-W_{k\dd}))-b(X^{(\dd)}_{k\dd})|^2\\
&=\int_{\R^d}\int_{\R^d}|b(y+z)-b(y)|^2 p^{(\dd)}(0,k\dd,x,y)\\
&\quad\times\ff{\exp(-\ff{1}{2(t-k\dd)}\<((\si^*\si)^{-1}(y)(z-b(y)(t-k\dd)),(z-b(y)(t-k\dd))\>)}{\ss{(2\pi)^d
\mbox{det}((t-k\dd)(\si\si^*)(y))}}\d y\d z\\
&\le  \ff{C_1}{(k\dd(t-k\dd))^{d/2}}
 \int_{\R^d}\int_{\R^d}|b(y+z)-b(y)|^2
\e^{-\ff{|z|^2}{4\hat\ll_0(t-k\dd)}} \e^{-\ff{
|x-y|^2}{\kk_0k\dd}}\d y\d z,
\end{split}
\end{equation}
for some constant $C_1>0$, where $\kk_0$ was introduced in \eqref{P2}.  
 With the aid of the fact that
\begin{equation}\label{S0}
\sup_{x\ge0}(x^\gg\e^{-\bb
x^2})=\Big(\ff{\gg}{2\e\bb}\Big)^{\ff{\gg}{2}},~~~\gg,\bb>0,
\end{equation}
we therefore infer from ({\bf A2}) and \eqref{D8} that
\begin{equation*}
\begin{split}
\mathbb{E}|b(X^{(\dd)}_t)-b(X^{(\dd)}_{k\dd})|^2&\le
 \ff{ C_2\phi(\kk_0k\dd)}{(t-k\dd)^{d/2}}\int_{\R^d}|z|^\aa\e^{-\ff{|z|^2}{4\hat\ll_0(t-k\dd)}}\d
z\\
&\le\ff{ C_3\phi(\kk_0k\dd)\dd^{\aa/2}}{
(t-k\dd)^{d/2}}\int_{\R^d}\e^{-\ff{|z|^2}{8\hat\ll_0(t-k\dd)}}\d
z\le C_4\phi(\kk_0k\dd)\dd^{\aa/2}
\end{split}
\end{equation*}
for some constants $C_2,C_3,C_4>0.$ Whence, we arrive at
\begin{equation}\label{F1}
\sum_{k=1}^{\lfloor
T/\dd\rfloor}\int_{k\dd}^{T\wedge(k+1)\dd}\mathbb{E}|b(X^{(\dd)}_t)-b(X^{(\dd)}_{k\dd})|^2\d
t\le C_4\dd^{\aa/2}\int_\dd^T\phi(\kk_0\lfloor t/\dd\rfloor\dd)\d t.
\end{equation}
Thus, \eqref{b-b} holds true by combining  \eqref{F2} with
\eqref{F1} and by utilizing that
$\phi:\R_+\to\R_+$ is locally integrable and continuous.
\end{proof}


 For any
$p,q\ge1$ and $0\le S\le T$, let $L^p_q(S,T)$ be the family of all
Borel measurable functions $f:[S,T]\times\R^d\to\R^d$ endowed with
the norm
\begin{equation*}
\|f\|_{L^p_q(S,T)}:=\bigg(\int_S^T\bigg(\int_{\R^d}|f_t(x)|^p\d
x\bigg)^{\ff{q}{p}}\d t\bigg)^{\ff{1}{q}}<\8.
\end{equation*}
For simplicity, in the sequel, we write $L^p_q(T)$ in place of
$L^p_q(0,T)$. Set
\begin{equation*}
\mathscr{K}_1:=\Big\{(p,q)\in(1,\8)\times(1,\8):\ff{d}{p}+\ff{2}{q}<2\Big\},~\mathscr{K}_2:=\Big\{(p,q)\in(1,\8)\times(1,\8):\ff{d}{p}+\ff{2}{q}<1\Big\}.
\end{equation*} Compared with  \eqref{1.1}, in \eqref{E1} we have
written the drift term as  $b(X_{t_\dd}^{(\dd)})$ in lieu of
$b(X_t^{(\dd)})$ so that the classical Krylov estimate (see e.g.
\cite{GM,HW18,KR,XZ,Z,Z2}) is unapplicable directly. However, the
following lemma manifests that
  $(X_t^{(\dd)})_{t\ge0}$ still satisfies  the Khasminskii estimate by employing Gaussian type estimate of  heat
  kernel although the Krylov estimate for
  $(X_{t_\dd}^{(\dd)})_{t\ge0}$ is invalid as Remark \eqref{Rem}
  below describes.

\beg{lem}\label{Kry}Assume $({\bf A1})$ and $({\bf A3})$. Then,  for
 $f\in L^p_q(T)$ with $(p,q)\in\mathscr{K}_1$,  the following Khasminskill  type estimate
\begin{equation}\label{D1}
\E\exp\Big(\ll\int_0^T|f_t(X^{(\dd)}_t)| \d t\Big)\le 2^{1+  T
(2\ll\aa_0\|f\|_{ L^p_q(T)})^{\gg_0}},~~~~\ll>0
\end{equation}
holds, where $\gg_0:=\ff{1}{1-1/q-d/2p}$ and
\begin{equation}\label{P0}
 \aa_0:=\ff{(1-1/p)^{\ff{d}{2}(1-1/p)}}{(\breve{\ll}_0(2\pi)^{\ff{1}{p}}
)^{\ff{d}{2}}}\Big\{\hat\ll_0^{\ff{d}{2}(1-1/p)}+\Lambda_3(\gg_0(1-1/q))^{\ff{q-1}{q}}(\kk_0/2
)^{\ff{d}{2}(1-1/p)}\Big\}.
\end{equation}
\end{lem}

\begin{proof}
For $0\le s\le t\le T$, note that
\begin{equation*}
\begin{split}
\E\Big(\int_s^t|f_r(X_r^{(\dd)})|\d
r\Big|\F_s\Big)&=\E\Big(\int_s^{t\wedge(s_\dd+\dd)}|f_r(X_r^{(\dd)})|\d
r\Big|\F_s\Big)+\E\Big(\int_{t\wedge(
s_\dd+\dd)}^t|f_r(X_r^{(\dd)})|\d r\Big|\F_s\Big)\\
&=:I_1(s,t)+I_2(s,t).
\end{split}
\end{equation*}
Since
\begin{equation*}
X_r^{(\dd)}=X_{s_\dd}^{(\dd)}+b(X_{s_\dd}^{(\dd)})(r-s_\dd)+\si(X_{s_\dd}^{(\dd)})(W_s-W_{s_\dd})+\si(X_{s_\dd}^{(\dd)})(W_r-W_s),~~r\in[s,s_\dd+\dd),
\end{equation*}
we derive from \eqref{F0} and H\"older's inequality that
\begin{equation}\label{H2}
\begin{split}
I_1(s,t)
&=\int_s^{t\wedge(s_\dd+\dd)}\int_{\R^d}f_r(y_{x,w}+z)\\
&\quad\times\ff{\exp\Big(-\ff{1}{2(r-s)}\<
(\si\si^*)^{-1}(x)(z-y_{x,w}),z-y_{x,w}\>\Big)}
{\ss{(2\pi(r-s))^d\mbox{det}((\si\si^*)(x))}}\d
z\Big|_{x=X_{s_\dd}^{(\dd)}}^{w=W_s-W_{s_\dd}}\d r\\
&\le\|f\|_{L_q^p(T)}\bigg(\int_s^{t\wedge(s_\dd+\dd)}\Big(\ff{1}{\ss{(2\pi(r-s))^d\mbox{det}((\si\si^*)(x))}}\\
&\quad\times\Big(\int_{\R^d}\exp\Big(-\ff{p}{2(p-1)(r-s)}\<
(\si\si^*)^{-1}(x)z,z\>\Big)\d
z\Big)^{\ff{p-1}{p}}\Big)^{\ff{q}{q-1}}\d
r\bigg)^{\ff{q-1}{q}}\Big|_{x=X_{s_\dd}^{(\dd)}}\\
&\le (2\pi
)^{-\ff{d}{2p}}((p-1)/p)^{\ff{d}{2}(1-\ff{1}{p})}(\hat\ll_0^{1-\ff{1}{p}}/\breve{\ll}_0)^{\ff{d}{2}}(t-s)^{\ff{1}{\gg_0}}\|f\|_{L_q^p(T)},
\end{split}
\end{equation}
where $y_{x,w}:=x+b(x)(r-s_\dd)+\si(x)w, x\in\R^d,w\in\R^m$, and
$\gg_0:=\ff{1}{1-1/q-d/2p}$. For $r>k\dd$, let
$X_{k\dd,r}^{(\dd),x}$ be the EM scheme determined by \eqref{E1}
with $X_{k\dd,k\dd}^{(\dd),x}=x$.  From the tower property of
conditional expectation, one has
\begin{equation*}
\begin{split}
I_2(s,t)&\le \int_{
s_\dd+\dd}^t\E\Big(|f_r(X_r^{(\dd)})|\Big|\F_s\Big)\d r=\int_{
s_\dd+\dd}^t\E\Big(\E\Big(|f_r(X_r^{(\dd)})|\Big|\F_{s_\dd+s}\Big)|\F_s\Big)\d
r\\
&=\int_{ s_\dd+\dd}^t\E\Big(
\E|f_r(X_{s_\dd+\dd,r}^{(\dd),x})|\Big|_{x=X_{s_\dd+\dd}^{(\dd)}}
\Big|\F_s\Big)\d r.
\end{split}
\end{equation*}
In terms of Lemma \ref{lem0}, besides  H\"older's inequality, one
obtains  that
\begin{equation*}
\begin{split}
\E|f_r(X_{s_\dd+\dd,r}^{(\dd),x})|&\le \ff{\LL_3}{(2\pi
\breve{\ll}_0(r-s_\dd-\dd))^{d/2}}\int_{\R^d}|f_r(y)|
 \e^{-\ff{|x-y|^2}{\kk_0(r-s_\dd-\dd)}}\d y\\
&\le\ff{\Lambda_3}{((2\pi)^{\ff{1}{p}}\breve{\ll}_0)^{d/2}}\Big(\ff{\kk_0(p-1)}{2p}\Big)^{\ff{d}{2}(1-1/p)}
(r-s_\dd-\dd)^{-\ff{d}{2p}}\bigg(\int_{\R^d}|f_r(y)|^p\d
y\bigg)^{\ff{1}{p}}.
\end{split}
\end{equation*}
This further yields by H\"older's inequality that
\begin{equation}\label{H3}
\begin{split}
I_2(s,t)&\le\ff{\Lambda_3}{((2\pi)^{\ff{1}{p}}\breve{\ll}_0)^{d/2}}\Big(\ff{\kk_0(p-1)}{2p}\Big)^{\ff{d}{2}(1-1/p)}\int_{
s_\dd+\dd}^t(r-s_\dd-\dd)^{-\ff{d}{2p}}\bigg(\int_{\R^d}|f_r(y)|^p\d
y\bigg)^{\ff{1}{p}}\d r\\
&=\ff{\Lambda_3\Big(\gg_0(1-
1/q)\Big)^{\ff{q-1}{q}}}{((2\pi)^{\ff{1}{p}}\breve{\ll}_0)^{d/2}}\Big(\ff{\kk_0}{2}(1-1/p)\Big)^{\ff{d}{2}(1-1/p)}
(t-s)^{\ff{1}{\gg_0}}\|f\|_{L_q^p(T)}.
\end{split}
\end{equation}
Hence, \eqref{H2} and \eqref{H3} imply
\begin{equation}\label{OO}
\E\Big(\int_s^t|f_r(X_r^{(\dd)})|\d r\Big|\F_s\Big)\le
\aa_0\,\|f\|_{L_q^p(T)} (t-s)^{\ff{1}{\gg_0}},~~~0\le s\le t\le T,
\end{equation}
in which $\aa_0>0$ was introduced in \eqref{P0}. For each $k\ge1$,
applying inductively \eqref{OO} gives
\begin{equation}\label{S3}
\begin{split}
&\E\bigg(\bigg(\int_s^t|f_r(X^{(\dd)}_r)| \d r\Big)^k\bigg|\F_s\bigg)\\
&=k! \E\Big(
\int_{\triangle_{k-1}(s,t)}|f_{r_1}(X^{(\dd)}_{r_1})|\cdots|f_{r_{k-1}}(X^{(\dd)}_{r_{k-1}})|\d r_1\cdots\d r_{k-1}\\
&\quad\times\E\Big(\int_{r_{k-1}}^t|f_k(X^{(\dd)}_{r_k})|\d
r_k\Big|\F_{r_{k-1}}\Big)\Big|\F_s\Big)\\
&\le \aa_0k! (t-s)^{\ff{1}{\gg_0}} \|f\|_{L^p_q(T)}  \\
&\quad\times\E\Big(\Big(
\int_{\triangle_{k-1}(s,t)}|f_{r_1}(X^{(\dd)}_{r_1})|\cdots|f_{r_{k-1}}(X^{(\dd)}_{r_{k-1}})|\d r_1\cdots\d r_{k-1}\Big)\Big|\F_s\Big)\\
&\le\cdots\le k!(\aa_0 (t-s)^{\ff{1}{\gg_0}}\|f\|_{ L^p_q(T)}
)^k,~~~~0\le s\le t\le T,
\end{split}
\end{equation}
where
\begin{equation*}
\triangle_k(s,t):=\{(r_1,\cdots,r_{k})\in\R^{k}:s\le r_1\le\cdots\le
r_{k}\le t\}.
\end{equation*}
Taking $\dd_0=(2\aa_0\ll\|f\|_{ L^p_q(T)})^{-\gg_0}$, one obviously
has
$
 \ll \aa_0\dd_0^{\ff{1}{\gg_0}}\|f\|_{ L^p_q(T)}=\ff{1}{2}.
$
With this and \eqref{S3} in hand, we derive that
\begin{equation}\label{D7}
\E\Big(\exp\Big(\ll\int_{(i-1)\dd_0}^{i\dd_0\wedge
T}|f_t(X^{(\dd)}_t)| \d
t\Big)\Big|\F_{(i-1)\dd_0}\Big)\le\sum_{k=0}^\8\ff{1}{2^k}=2,~~~i\ge1,
\end{equation}
which further implies inductively that
\begin{equation}\label{f2}
\begin{split}
\E\exp\bigg(\ll\int_0^T|f_t(X^{(\dd)}_t)| \d t\bigg)
&=\E\bigg(\exp\bigg(\ll\sum_{i=1}^{\lfloor
T/\dd_0\rfloor}\int_{(i-1)\dd_0}^{ i\dd_0 }|f_t(X^{(\dd)}_t)| \d
t\bigg)\\
&\quad\times \E\bigg(\exp\bigg(\ll\int_{\lfloor
T/\dd_0\rfloor\dd_0}^{T }|f_t(X^{(\dd)}_t)| \d
t\bigg)\Big|\F_{\lfloor T/\dd_0\rfloor\dd_0}\bigg)\bigg)\\
&\le2\,\E\exp\bigg(\ll\sum_{i=1}^{\lfloor
T/\dd_0\rfloor}\int_{(i-1)\dd_0}^{ i\dd_0 }|f_t(X^{(\dd)}_t)| \d
t\bigg)\\
&\le\cdots\le 2^{1+T/\dd_0  }.
\end{split}
\end{equation}
Therefore, \eqref{D1} is now available by recalling
$\dd_0=(2\aa_0\ll\|f\|_{ L^p_q(T)})^{-\gg_0}$.
\end{proof}


The following lemma is concerned with Khasminskill's estimate for
the solution process $(X_t)_{t\ge0}$, which is more or less
standard; see, for instance, \cite{GM,HW18,KR,XZ,Z,Z2}. Whereas, we
herein state the Khasminskill estimate and provide a sketch of its
proof merely   for the sake of   explicit upper bound.

\begin{lem}
Assume $({\bf A1})$ and $({\bf A3})$. Then,  for
 $f\in L^p_q(T)$ with $(p,q)\in\mathscr{K}_1$ and $\ll>0,$
\begin{equation}\label{f1}
\E\exp\Big(\ll\int_0^T|f_t(X_t)| \d t\Big)\le 2^{1+  T (2\ll\hat
\aa_0\|f\|_{ L^p_q(T)})^{\gg_0}},
\end{equation}
where
\begin{equation}\label{f6}
\hat \aa_0:=(2\pi
)^{-\ff{d}{2p}}\hat\bb_T(8(p-1)/p)^{\ff{d}{2}(1-\ff{1}{p})}(\hat\ll_0^{1-\ff{1}{p}}/\breve{\ll}_0)^{\ff{d}{2}},~~~~
\hat\bb_T:=\e^{\ff{\|b\|_\8^2T}{2\hat\ll_0}}\sum_{i=0}^\8\ff{
\bb_T^i}{\GG(1+\ff{i}{2})}
\end{equation}
with $\bb_T$ being given in \eqref{A00}.
\end{lem}

\begin{proof}
By \eqref{A0} below, it follows from H\"older's inequality and
Markov property that
\begin{equation}\label{f3}
\begin{split}
\E\Big(\int_s^t|f_r(X_r)|\d
r\Big|\F_s\Big)&=\int_s^t(\E|f_r(X_r^{s,x})|)\d r\Big|_{x=X_s}\\
&\le\hat\bb_T \int_s^t\int_{\R^d}|f_r(y)|\ff{\e^{-\ff{
|y-x|^2}{16\hat\ll_0(r-s)}}}{(2\pi\breve{\ll}_0(r-s))^{d/2}}\d y\d
r\Big|_{x=X_s}\\
&\le\hat \aa_0(t-s)^{1-\ff{d}{2p} -\ff{1}{q}}\|f\|_{L_q^p(T)},
\end{split}
\end{equation}
where $(X^{s,x}_t)_{t\ge s}$ stands for the solution to \eqref{1.1}
with the initial value $X_s^{s,x}=x,$ and  $\hat\bb_T,\hat \aa_0>0$
were introduced in \eqref{f6}. Then,  \eqref{f1} follows immediately
by  utilizing \eqref{f3}  and
 by following the exact line  to derive \eqref{f2}.
\end{proof}

\begin{rem}\label{Rem}
{\rm In \eqref{OO},  Krylov's estimate for $(X_t^{(\dd)})_{t\ge0}$
instead of  $(X_{t_\dd}^{(\dd)})_{t\ge0}$ is available. Whereas, the
Krylov estimate associated with $(X_{t_\dd}^{(\dd)})_{t\ge0}$ no
longer holds true. Indeed, if we take $s,t\in[k\dd,(k+1)\dd)$ for
some integer $k\ge1,$ we obviously have
\begin{equation}\label{H1}
\E\Big(\int_s^t|f_{r_\dd}(X_{r_\dd}^{(\dd)})|\d r\Big|\F_s\Big)=
|f_{k\dd}(X_{k\dd}^{(\dd)})|(t-s),  ~~ f\in L^p_q(T),
~(p,q)\in\mathscr{K}_1
\end{equation}
which is a random variable. Hence, it is impossible to control the
quantity on the left hand side of  \eqref{H1} by $\|f\|_{L^p_q(T)}$
up to a constant. Moreover,   we would like to refer to e.g.
\cite{shao} for more details.}
\end{rem}

Before we go further, we  introduce some additional  notation. For
$p\ge1$ and $m\ge0$, let $H^m_p$ be the usual Sobolev space on
$\R^d$ with the norm
\begin{equation*}
\|f\|_{H_p^m}:=\sum_{k=0}^m\|\nn^mf\|_{L^p},
\end{equation*}
where  $\nn^m$ denotes  the $m$-th order gradient operator. For
$m\ge0$ and $0\le S\le T$, let
$\mathbb{H}_p^{m,q}(S,T)=L^q(S,T;H_p^m)$ and
$\mathscr{H}_p^{m,q}(S,T)$ be the collection of all functions
$f:(S,T)\times\R^d\to\R$ such that $u\in\mathbb{H}_p^{m,q}(S,T)$ and
$\partial_t f\in L^p_q(S,T)$. For a locally integrable function
$h:\R^d\to\R,$ the Hardy-Littlewood maximal operator $\mathscr{M}h$
is defined as below
\begin{equation*}
(\mathscr{M}h)(x)=\sup_{r>0}\ff{1}{|B_r(x)|}\int_{B_r(x)}h(y)\d
y,~~~~x\in\R^d,
\end{equation*}
where $B_r(x) $  is the ball with the radius $r$ centered at the
point $x$ and $|B_r(x)|$ denotes the d-dimensional Lebesgue measure
of $B_r(x)$.

To make the content self-contained, we recall the Hardy-Littlewood
maximum theorem (see e.g. \cite[Lemma 5.4]{Z2}), which is stated as
the lemma below.

\begin{lem}
For any continuous and weak differential function $f:\R^d\to\R$,
there exists a constant $C>0$ such that
\begin{equation}\label{S1}
|f(x)-f(y)|\le C|x-y|\{(\mathscr{M}|\nn f|)(x)+(\mathscr{M}|\nn
f|)(y)\},~~~\mbox{ a.e. } x,y\in\R^d.
\end{equation}
Moreover, for any $f\in L^p(\R^d),p>1$, there exists a constant
$C_p,$ independent of $d$, such that
\begin{equation}\label{S2}
\|\mathscr{M}f\|_{L^p}\le C_p\|f\|_{L^p}.
\end{equation}
\end{lem}

Now we are in position to complete
\begin{proof}[{\bf Proof of Theorem \ref{th1}}]
For any $\lambda >0$, consider the following PDE for
$u^\ll:[0,T]\times\R^d\to\R^d:$ \beq\label{PDE}\beg{split}\partial_t
u^\ll+\frac{1}{2}\sum_{i,j=1}^{d}{\langle
\sigma\sigma^{\ast}e_{i},e_{j}\rangle}\nabla_{e_{i}}\nabla_{e_{j}}u^\ll+b+\nabla_{b}u^\ll=\lambda
u^\ll,
\end{split}\end{equation}
where $(e_j)_{1\le j\le d}$ stipulates   the orthogonal basis of
$\R^d.$ According to \cite[Lemma 4.2]{XZ},    \eqref{PDE} has a
unique solution $u^{\lambda}\in \mathscr{H}_{2p}^{2,2q}(0,T)$ for
the pair $ (p,q)\in\mathscr{K}_1$ due to $p>\ff{d}{2}$ satisfying
\begin{equation}\label{F3}
 (1\vee\ll)^{1 -\ff{d}{4p}-\ff{1}{2q}} \|\nn u^\ll\|_{T,\8}
 +\|\nn^2 u^\ll\|_{L^{2p}_{2q}(T)}\leq C_1\||b|^2\|_{L^{p}}
\end{equation}
for some constant $C_1>0$, where $ \|\nn u^\ll\|_{T,\8}:=\sup_{0\le
t\le T,x\in\R^d}\|\nn u_t^\ll(x)\|_{\rm HS}$. With the help of
\eqref{F3}, there  is a constant $\ll_0\ge1$ such that
\begin{equation}\label{B1}
 \|\nn u^\ll\|_{T,\8}\le\ff{1}{2},~~~\ll\ge\ll_0.
\end{equation}
 For $u^\ll\in
{\mathscr{H}_{2p}^{2,2q}(0,T)}$, there exists a sequence
$u^{\ll,k}\in C^{1,2}([0,T]\times\R^d)$ such that
\begin{equation*}
\lim_{k\to\infty}\|u^{\lambda,k}-u^\lambda\|_{{\mathscr{H}_{2p}^{2,2q}(0,T)}}=0,
 \end{equation*}
where
\begin{equation*}
\|u\|_{{\mathscr{H}_{2p}^{2,2q}(0,T)}}:=\|\partial_t
u\|_{L^{2p}_{2q}(0,T)}+\| u\|_{\mathbb{H}_{2p}^{2,2q}(0,T)}.
\end{equation*}
Henceforth, we can apply directly It\^o's  formula to $u^\ll\in
{\mathscr{H}_{2p}^{2,2q}(0,T)}$ by adopting a standard approximation
approach; see e.g. the arguments of \cite[Theorem 2.1]{XZ} and
\cite[Lemma 4.3]{Z2} for more details. Set
$\theta^\ll_t(x):=x+u^\ll_t(x), x\in\R^d,$ and
$Z_t^{(\dd)}:=X_t-X^{(\dd)}_t$. By It\^o's formula, we obtain from
\eqref{PDE} that
\begin{equation*}
\begin{split}
\d\theta^\ll_t(X_t)&= \ll u^\ll(X_t)\d t+
\nn\theta^\ll_t(X_t)\sigma(X_t)\d W_t\\
\d\theta^\ll_t(X^{(\dd)}_t)&=\Big \{\ll u^\ll(X^{(\dd)}_t) +
\nn\theta^\ll_t(X^{(\dd)}_t)( b(X^{(\dd)}_{t_\dd})-b(X^{(\dd)}_t))
+\ff{1}{2}\sum_{i,j=1}^d\<((\si\si^*)(X_{t_\dd}^{(\dd)})\\
&\quad-(\si\si^*)(X_t^{(\dd)}))e_i,e_j\>\nn_{e_i}\nn_{e_j}u_t^\ll(X_t^{\dd}))
\Big\}\d t + \nn
\theta^\ll_t(X^{(\dd)}_t)\sigma(X^{(\dd)}_{t_\dd})\d W_t.
\end{split}
\end{equation*}
Applying It\^o's formula once more, and taking advantage of
 \begin{equation}\label{F4}
 \ff{1}{4}|Z_t^{(\dd)}|^2\le  |\theta^\ll_t(X_t)-\theta^\ll_t(X^{(\dd)}_t)|^2\le \ff{5}{2}|Z_t^{(\dd)}|^2
 \end{equation}
due to \eqref{B1} yields that
\begin{equation}\label{F5}
\begin{split}
|Z_t^{(\dd)}|^2&\le8\ll\int_0^t\<\theta^\ll_s(X_s)-\theta^\ll_s(X^{(\dd)}_s),u^\ll(X_s)-u^\ll(X^{(\dd)}_s)\>\d
s\\
&\quad+8\int_0^t \<\theta^\ll_s(X_s)-\theta^\ll_s(X^{(\dd)}_s),\nn
\theta^\ll_s(X^{(\dd)}_s)(b(X^{(\dd)}_s)-b(X^{(\dd)}_{s_\dd})) \>\d
s\\
&\quad+\sum_{i,j=1}^d\int_0^t\<((\si\si^*)(X_{s_\dd}^{(\dd)})-(\si\si^*)(X_s^{(\dd)}))e_i,e_j\>\<\theta^\ll_s(X_s)-\theta^\ll_s(X^{(\dd)}_s),
\nn_{e_i}\nn_{e_j}u_s^\ll(X_s^{\dd}))\>\d s\\
&\quad+4 \int_0^t\|\nn\theta^{\ll}_s(X_s)\si(X_s)-\nn\theta^{\ll}_s(X^{(\dd)}_s)\si(X^{(\dd)}_{s_\dd})\|^2_{\rm  HS}\d s+M_t\\
&=: I_{1,\dd}(t)+I_{2,\dd}(t)+I_{3,\dd}(t)+I_{4,\dd}(t)+M_t,
\end{split}
\end{equation}
where
\begin{equation*}
M_t:=
8\int_0^t\Big\<\theta^\ll_s(X_s)-\theta^\ll_s(X^{(\dd)}_s),((\nn\theta^\ll_s\si)(X_s)-\nn\theta^\ll_s(X^{(\dd)}_s)\si(X_{s_\dd}^{(\dd)}))\d
W_s\Big\>.
\end{equation*}
By means of \eqref{B1},  we have
\begin{equation}\label{F6}
I_{1,\dd}(t)\le 6\ll\int_0^t|Z_s^{(\dd)}|^2\d s.
\end{equation}
Also, by virtue of \eqref{B1}, besides \eqref{F4}, we find that
there exists a constant $C_2>0$ such that
\begin{equation}\label{F7}
I_{2,\dd}(t)\le C_2\Big\{\int_0^t|Z_s^{(\dd)}|^2\d
s+\int_0^t|b(X^{(\dd)}_s)-b(X^{(\dd)}_{s_\dd})|^2\d s\Big\}.
\end{equation}
Next, with the aid of \eqref{F0}, \eqref{F00}, and \eqref{F4}, we
infer that
\begin{equation}
\begin{split}
I_{3,\dd}(t)&\le
C_3\int_0^t\|\si(X_s^{(\dd)})-\si(X_{s_\dd}^{(\dd)})\|_{\rm
HS}|Z_s^{(\dd)}|\cdot\| \nn^2 u_s^\ll(X_s^{\dd})\|_{\rm HS}\d s\\
&\le L_0C_3\int_0^t | X_s^{(\dd)}- X_{s_\dd}^{(\dd)}
|\cdot|Z_s^{(\dd)}|\cdot\| \nn^2 u_s^\ll(X_s^{\dd})\|_{\rm HS}\d s\\
&\le\ff{L_0C_3}{2}\int_0^t\{\| \nn^2 u_s^\ll(X_s^{\dd})\|_{\rm
HS}^2|Z_s^{(\dd)}|^2+| X_s^{(\dd)}- X_{s_\dd}^{(\dd)} |^2\}\d s.
\end{split}
\end{equation}
for some constant $C_3>0.$ Furthermore,  thanks to \eqref{F0},
\eqref{F00}, \eqref{S1} and \eqref{B1},  we derive from  H\"older's
inequality  that
\begin{equation}\label{F8}
\begin{split}
I_{4,\dd}(t)&\le C_4 \int_0^t|Z_s^{(\dd)}|^2\Big\{(\mathscr{M}
\|\nn^2u^{\ll}_s\|_{\rm
HS}^2)(X_s)+(\mathscr{M}\|\nn^2u^{\ll}_s\|_{\rm HS}^2)(X^{(\dd)}_s)
\Big\} \d s\\
&\quad+C_4\int_0^t| X_s^{(\dd)}- X_{s_\dd}^{(\dd)} |^2\d s
\end{split}
\end{equation}
for some constant $C_4>0$.  As a result, plugging
\eqref{F6}-\eqref{F8} into \eqref{F5} gives that
\begin{equation*}\label{E6}
|Z_t^{(\dd)}|^2\le\int_0^t|Z_s^{(\dd)}|^2\d
A_s+\int_0^t\Big\{C_2|b(X^{(\dd)}_s)-b(X^{(\dd)}_{s_\dd})|^2+\ff{L_0C_3+2C_4}{2}|
X_s^{(\dd)}- X_{s_\dd}^{(\dd)} |^2\Big\}\d s+M_t,
\end{equation*}
in which,  for some constant $C_5>0,$
\begin{equation*}
\begin{split}
A_t:= C_5\int_0^t\Big\{&1+(\mathscr{M} \|\nn^2u^{\ll}_s\|_{\rm
HS}^2)(X_s)+(\mathscr{M}\|\nn^2u^{\ll}_s\|_{\rm HS}^2)(X^{(\dd)}_s)+
 \|\nn^2u_s^{\ll}\|^2_{\rm HS}(X_s^{\dd})\Big\}\d s.
\end{split}
\end{equation*}
Consequently, we  deduce by stochastic Gronwall's inequality (see
e.g. \cite[Lemma 3.8]{XZ}) that, for $0<\kk'<\kappa<1$,
\begin{equation*}
\begin{split}
\Big(\E\Big(\sup_{0\le s\le
t}|Z_t^{(\dd)}|^{2\kk'}\Big)\Big)^{1/\kk'}&\le
\Big(\ff{\kk}{\kk-\kk'}\Big)^{1/\kk'}\Big(\E\e^{\kk
A_t/(1-\kk)}\Big)^{(1-\kk)/\kk}\\
&\quad\times\int_0^t\Big\{C_2\E|b(X^{(\dd)}_s)-b(X^{(\dd)}_{s_\dd})|^2+\ff{L_0C_3+2C_4}{2}\E|
X_s^{(\dd)}- X_{s_\dd}^{(\dd)} |^2\Big\}\d s.
\end{split}
\end{equation*}
This, together with \eqref{b-b} and
\begin{equation*}
\sup_{0\le t\le T}\E| X_t^{(\dd)}- X_{t_\dd}^{(\dd)} |^2\le C_6\dd
\end{equation*}
 for some constant $C_6>0$, leads to
\begin{equation}\label{A22}
\begin{split}
\Big(\E\Big(\sup_{0\le s\le
t}|Z_t^{(\dd)}|^{2\kk'}\Big)\Big)^{1/\kk'}&\le C_7\Big(\E\,\e^{\kk
A_t/(1-\kk)}\Big)^{(1-\kk)/\kk}(\dd+\dd^{\aa/2})
\end{split}
\end{equation}
for some constant $C_7>0.$ By H\"older's inequality, we deuce that
for some constant $C_8>0,$
\begin{equation*}
\begin{split}
\E\e^{\ff{\kk A_t}{1-\kk}}
&\le\e^{\ff{\kk C_5t}{1-\kk}}\bigg(\exp\bigg(C_8 \int_0^t
(\mathscr{M}
\|\nn^2u^{\ll}_s\|^2_{\rm HS})(X_s)\d s\bigg)\bigg)^{1/2}\\
&\quad\times\bigg(\E\exp\bigg(C_8\int_0^t
(\mathscr{M}\|\nn^2u^{\ll}_s\|^2_{\rm HS})(X^{(\dd)}_s)\d
s\bigg)\bigg)^{1/4}\\
&\quad\times\bigg(\E\exp\bigg(C_8\int_0^t \|\nn^2u_s^{\ll}\|^2_{\rm
HS}(X_s^{\dd})\d s\bigg)\bigg)^{1/4}.
\end{split}
\end{equation*}
This, in addition to \eqref{D1}, \eqref{f1},  \eqref{f3} as well as
\eqref{F3}, implies that
\begin{equation}\label{A23}
\begin{split}
E\e^{\ff{\kk A_t}{1-\kk}}&\le \exp\Big(C_9\Big(1+\|\|\nn^2u^{\ll}
\|^2_{\rm HS}\|_{ L^{p}_{q}(T)}^{\gg_0}+\|\mathscr{M}\|\nn^2u^{\ll}
\|^2_{\rm HS}\|_{ L^{p}_{q}(T)}^{\gg_0}\Big)\Big)\\
&\le\exp\Big(C_{10}\Big(1+\|\nn^2u^{\ll} \|_{
L^{2p}_{2q}(T)}^{2\gg_0}
\Big)\Big)\\
&\le\exp\Big(C_{11}\Big(1+\||b|^2\|_{ L^{p}}^{\gg_0} \Big)\Big)
\end{split}
\end{equation}
for some constants $C_9,C_{10},C_{11}>0$. Thus, the  assertion
\eqref{W1} follows from \eqref{A22} and \eqref{A23}.
\end{proof}

\section{Proof of Theorem \ref{th2}}\label{sec3}
In this section, we aim to complete the proof of Theorem \ref{th2}
by carrying out a truncation approach; see, for example,
\cite{BHY,NT2} for further details.

  Let $\psi:\R_+\to[0,1]$
be a smooth function
  such that
  $\psi(r)=1,r\in[0,1]$, and $\psi(r)\equiv0, r\ge2.$ For each
  integer $k\ge1$, let $b_k(x)=b(x)\psi(|x|/k)$, $x\in\R^d$, be the truncation function associated with the drift $b$.
  A
 direct calculation shows that
\begin{equation}\label{D2}
\|b_k\|_\8\le\|b\|_\8~~~~\mbox{ and
}~~~~\||b_k|^2\|_{L^p}\le\Big(\ff{2^d\pi^{\ff{d}{2}}}{\GG(\ff{d}{2}+1)}\Big)^{1/p}k^{\ff{d}{p}}\|b\|_\8^2.
\end{equation}
Consider the following truncated SDE corresponding to \eqref{1.1}
\begin{equation}\label{D3}
\d X_t^{k}=b_k(X_t^{k})\d t+\si(X_t^k)\d W_t,~~~t\ge0,~~X_0^{k}=X_0.
\end{equation}
  The EM
scheme concerned with \eqref{D3} is given by
\begin{equation*}
\d X_t^{k,(\dd)}=b_k(X_{t_\dd}^{k,(\dd)})\d
t+\si(X_{t_\dd}^{k,(\dd)})\d W_t,~~~t\ge0,~~X_0^{k,(\dd)}=X_0^{(k)}.
\end{equation*}
For $q\in(0,2)$, observe that
\begin{equation}
\begin{split}
\E  \|X -X^{(\dd)} \|^q_{T,\8} &\le
3^{0\vee(q-1)} \{\E\|X-X^{k}\|^q_{T,\8}+\E\|X^{(\dd)}-X^{k,(\dd)}\|^q_{T,\8}\\
&\quad+\E\|X_{t}^k-X^{k,(\dd)}\|^q_{T,\8} \}\\
&=:3^{0\vee(q-1)}\{I_1+I_2+I_3\}.
\end{split}
\end{equation}
Via H\"older's inequality and $\{X_t\neq X_t^{k},0\le t\le
T\}\subseteq\{\|X \|_{T,\8}\ge k\}$, it follows that
\begin{equation*}
\begin{split}
I_1&=\E\Big( \|X-X^{k}\|^q_{T,\8}{\bf1}_{\{ \|X \|_{T,\8}\ge
k\}}\Big)\le\Big(\E\Big(\|X-X^{k}\|_{T,\8}^{2q}\Big)\Big)^{1/2}\Big(\P(\|X
\|_{T,\8}\ge k)\Big)^{1/2}.
\end{split}
\end{equation*}
Since
\begin{equation*}
\|X \|_{T,\8}\le |x|+\|b\|_\8T+\sup_{0\le t\le T} |M_t|,
\end{equation*}
in which  $M_t:=\int_0^t\si(X_s)\d W_s, t\ge0,$ with the quadratic
variation $\<M\>_T\le d\hat\ll_0 T,$ we derive from
\cite[Proposition 6.8, p147]{Shi}  that
\begin{equation}\label{D4}
\begin{split}
 \P(\|X \|_{T,\8}\ge k)&\le\P\Big(\sup_{0\le t\le
T}|M_t|\ge k-|x|-\|b\|_\8T,\<M\>_T\le d\hat\ll_0 T\Big)\\
&\le 2d\exp\Big(-\ff{(k-|x|-\|b\|_\8T)^2}{4d^2\hat\ll_0T}\Big)\\
&\le 2d\exp\Big(\ff{(|x|+\|b\|_\8T)^2
}{4d^2\hat\ll_0T}\Big)\e^{-\ff{k^2}{8d^2\hat\ll_0T}},
\end{split}
\end{equation}
where in the last display we used the inequality: $(a-b)^2\ge
a^2/2-b^2,a,b\in\R.$ \eqref{D4}, in addition to
\begin{equation*}
\E \|X\|_{T,\8}^{2q}+\E \| X^{k}\|_{T,\8}^{2q}\le C_{1}
\end{equation*}
for some constant $C_{1}$  yields
\begin{equation}\label{D5}
I_1\le C_2\exp\Big(\ff{(|x|+\|b\|_\8T)^2
}{8d^2\hat\ll_0T}\Big)\e^{-\ff{k^2}{16d^2\hat\ll_0T}}
\end{equation}
for some constant $C_2>0$. Following a similar procedure to derive
\eqref{D5}, we also derive that
\begin{equation}\label{D6}
I_2\le C_3\exp\Big(\ff{(|x|+\|b\|_\8T)^2
}{8d^2\hat\ll_0T}\Big)\e^{-\ff{k^2}{16d^2\hat\ll_0T}}
\end{equation}
for some constant $C_3>0.$ Moreover, for $p,q>2$ with
$\ff{d}{p}+\ff{1}{q}<1$, according to Theorem \ref{th1}, there exist
constants $C_4,C_5>0$ such that
\begin{equation*}
\E\|X_{t}^k-X^{k,(\dd)}\|^q_{T,\8}\le C_4 \e^{C_5\||b_k|^2\|_{
L^{p}}^{\gg_0} } \dd^{\ff{q}{2}(1\wedge\ff{\aa}{2})}.
\end{equation*}
This, together with \eqref{D2}, implies
\begin{equation}\label{D9}
\E\|X_{t}^k-X^{k,(\dd)}\|^q_{T,\8}\le C_4 \e^{C_6\|b\|_\8^{2\gg_0}
k^{\ff{d\gg_0}{p}} } \dd^{\ff{q}{2}(1\wedge\ff{\aa}{2})}
\end{equation}
for some constant $C_6>0.$ As a consequence, from \eqref{D5},
\eqref{D6}, and \eqref{D9}, we arrive at
\begin{equation*}
\E\Big(\sup_{0\le t\le T}|X_{t}-X^{(\dd)}_t|^q\Big)\le C_8\Big\{
\e^{-\ff{k^2}{16d^2\hat\ll_0T}}+\e^{C_7  k^{\ff{d\gg_0}{p}} }
\dd^{\ff{q}{2}(1\wedge\ff{\aa}{2})}\Big\}
\end{equation*}
for some constants $C_7,C_8>0$. Thereby, the desired assertion
\eqref{S4} follows by taking
\begin{equation*}
k=\Big(-8qd^2\hat\ll_0T\Big(1\wedge\ff{\aa}{2}\Big)\log\dd\Big)^{\ff{1}{2}}.
\end{equation*}

\section{Illustrative Examples}\label{sec4}
In this section, we intend to give  examples to demonstrate that the
assumption imposed on drift term   holds true.

\begin{exa}\label{exa}
{\rm Let $b(x)={\bf1}_{[a_1,a_2]}(x), x\in\R,$ for some constants
$a_1<a_2$. Apparently, $b$ is not continuous at all but $b^2\in L^p$
for any $p\ge1$. Observe that
\begin{equation*}
\lim_{\vv\downarrow0}
\ff{-\vv(b(a_1-\vv)-b(a_1))}{\vv^2}=\lim_{\vv\downarrow0}
\ff{1}{\vv}=\8
\end{equation*}
so that $b$ does not obey the one-side Lipschitz condition. Next we
aim to show that $b$ given above satisfies ({\bf A2}). By a direct
calculation, for any $s>0$ and $y\in\R,$
\begin{equation*}
\begin{split}
\int_{-\8}^\8|b(x+y+z)-b(x+y)|^2\e^{-\ff{x^2}{s}}\d x
&\le\int_{-\8}^\8|b(x+z)-b(x)|^2 \d
x\\&=\int_{a_1-z}^{a_2-z}{\bf1}_{[a_1,a_2]^c}(x) \d
x+\int_{a_1}^{a_2}{\bf1}_{[a_1-z,a_2-z]^c}(x) \d x\\
&=:I_1(z)+I_2(z).
\end{split}
\end{equation*}
If $z\ge0,$ then
\begin{equation*}
I_1(z)=\int_{a_1-z}^{(a_2-z)\wedge a_1} \d x\le |z|~~~\mbox{ and }
~~~I_2(z)=\int^{a_2}_{(a_2-z)\vee a_1} \d x\le |z|.
\end{equation*}
On the other hand, for $z<0$, we have
\begin{equation*}
I_1(z)=\int_{(a_1-z)\vee a_2}^{a_2-z} \d x\le |z|~~~\mbox{ and }
~~~I_2(z)=\int^{a_2\wedge(a_1-z)}_{ a_1} \d x\le |z|.
\end{equation*}
In a word, we conclude that ({\bf A2}) holds with $\aa=1$ and
$\phi(s)=s^{-\ff{1}{2}}$ therein.}
\end{exa}

\begin{exa}\label{exa1}
{\rm   For $\theta>0$ and $p\in[2,\8)\cap(d,\8)$,  if the Gagliardo
seminorm
\begin{equation*}
[b]_{W^{p,\theta}}:=\Big(\int_{\R^d\times\R^d}\ff{|b(x)-b(y)|^p}{|x-y|^{d+p\theta}}\d
x\d y\Big)^{\ff{1}{p}}<\8,
\end{equation*}
then $b\in GB^2_{1-\ff{d}{p},\theta}(\R^d)$.
  Indeed, by H\"older's inequality and
\eqref{S0}, it follows that
\begin{equation}\label{B3}
\begin{split}
&\ff{1}{(rs)^{d/2}}\int_{\R^d\times\R^d}|b(x)-b(y)|^2\e^{-\ff{|x-z|^2}{s}}\e^{-\ff{|y-x|^2}{r}}\d
y\d x\\
&=\ff{1}{(rs)^{d/2}}\int_{\R^d\times\R^d}\ff{|b(x)-b(y)|^2}{|x-y|^{\ff{2d}{p}+2\theta}}\e^{-\ff{|x-z|^2}{s}}
\e^{-\ff{|y-x|^2}{r}}|x-y|^{\ff{2d}{p}+2\theta}\d
y\d
x\\
&\le C_1
  \ff{[b]_{W^{p,\theta}}^{\ff{2}{p}}}{(rs)^{d/2}}\Big(\int_{\R^d\times\R^d}
\e^{-\ff{p|x-z|^2}{(p-2)s}} \e^{-\ff{p|x-y|^2}{(p-2)r}}
|x-y|^{\ff{2(d+p\theta)}{p-2} }\d y\d x\Big)^{\ff{p-2}{p}}\\
&\le C_2[b]_{W^{p,\theta}}^{\ff{2}{p}}\ff{r^{\ff{d}{p}
+\theta}}{(rs)^{d/2}}\Big(\int_{\R^d\times\R^d}
\e^{-\ff{p|x-z|^2}{(p-2)s}} \e^{-\ff{p|x-y|^2}{2(p-2)r}}  \d y\d
x\Big)^{\ff{p-2}{p}}\\
&\le
C_3[b]_{W^{p,\theta}}^{\ff{2}{p}}\Big(\ff{p-2}{p}\Big)^{\ff{d(p-2)}{p}}s^{-\ff{d}{p}}r^\theta,~~~r,s>0,
z\in\R^d,p>2
\end{split}
\end{equation}
for some constants $C_1,C_2,C_3>0$. On the other hand, if $d=1$ and
$p=2$, we deduce from \eqref{B3} that $b\in GB^2_{1/2,\theta}(\R^d)$
due to $\lim_{x\to0}x^x=1.$ }
\end{exa}

\begin{exa}\label{exa3}
{\rm For $0<a<b<\8,$ $f(\cdot):=\mathds{1}_{[a,b]}(\cdot)\in
GB^2_{\ff{1}{2},\ff{1}{2}}(\R^d)$ whereas $f\notin
W^{\ff{1}{2},2}(\R)$. In fact, it is easy to see that
\begin{equation*}
f\in\cap_{0\le\theta<\ff{1}{2}}W^{\theta,2},~~~~~\lim_{\theta\uparrow\ff{1}{2}}[f]_{W^{\theta,2}}=\8,
\end{equation*}
which yields $f\notin W^{\ff{1}{2},2}(\R)$. On the other hand, since
\begin{equation*}
\ff{1}{(rs)^{d/2}}\int_{\R^2}|f(x)-f(y)|^2\e^{-\ff{|x-z|^2}{s}}\e^{-\ff{|y-x|^2}{r}}\d
y\d x\le Cs^{-\ff{1}{2}}r^{\ff{1}{2}},~~~r,s>0, z\in\R
\end{equation*}
for some constant $C>0$, we arrive at $f \in
GB^2_{\ff{1}{2},\ff{1}{2}}(\R^d)$.

}
\end{exa}

\section{Appendix}\label{sec5}
The lemma below provides   explicit estimates of the coefficients
concerning Gaussian type estimate of transition density for the
diffusion process $(X_t)_{t\ge0}$ solving \eqref{1.1}.

\begin{lem}\label{AA}
Under $\|b\|_\8<\8$ and $({\bf A3})$,  the transition density $p$ of
$(X_t)_{t\ge s}$ satisfies
\begin{equation}\label{A0}
p(s,t,x,x')\le\e^{\ff{\|b\|_\8^2T}{2\hat\ll_0}}\sum_{i=0}^\8\ff{
\bb_T^i}{\GG(1+\ff{i}{2})}\,p_0(t-s,x,x'),~~~0\le s< t\le T,
x,x'\in\R^d,
\end{equation}
where $\GG(\cdot)$ is the Gamma function,  and
\begin{equation}\label{A00}
\begin{split}
&\bb_T:=2^{3d+1}\Big(\ff{\hat\ll_0}{\breve{\ll}_0}\Big)^{d+1}(\pi
T)^{\ff{1}{2}}\Big\{ \ff{\|b\|_\8}{\ss{\hat{\ll}_0}}+L_0
 (
d+2\ss d) \Big\}\e^{\ff{\|b\|_\8^2T}{4\hat\ll_0}},
p_0(t,x,x'):=\ff{\e^{-\ff{
|x-x'|^2}{16\hat\ll_0t}}}{(2\pi\breve{\ll}_0t)^{d/2}}.
\end{split}
\end{equation}
\end{lem}

\begin{proof}
The proof of Lemma \ref{AA} is based on the parametrix method
\cite{KM,LM}. To complete the proof of Lemma \ref{AA}, it suffices
to refine the argument of \cite[Lemma 3.2]{KM}; see also e.g.
\cite[p1660-1662]{LM} for further details. Under $\|b\|_\8<\8$ and
$({\bf A3})$,  $X_t$ admits a smooth transition density $p(s,t,x,y)$
at the point $y$, given $X_s=x$, such that
\begin{equation}\label{A1}
\begin{split}
\pp_tp(s,t,x,y)&=L^*p(s,t,x,y),~~~p(s,t,x,\cdot)=\dd_x(\cdot),~~t\downarrow s,\\
\pp_sp(s,t,x,y)&=-Lp(s,t,x,y),~~~p(s,t,\cdot,y)=\dd_y(\cdot),~~s\uparrow
t,
\end{split}
\end{equation}
where $L$ is the infinitesimal generator of \eqref{1.1} and $L^*$ is
its   adjoint operator. For  $t>s$ and $x,x'\in\R^d$, let $ \tt
X^{s,x,x'}_t $ solve  the   frozen SDE
\begin{equation}\label{A2}
\d \tt X_t^{s,x,x'}=b(x')\d t+\si(x')\d W_t, ~~t>s, ~\tt
X_s^{s,x,x'}=x\in\R^d
\end{equation}
and $\tt p^{x'}(s,t,x,x')$ stand  for its transition density at
$x'$, given $\tt X^{s,x,x'}_s=x.$ Apparently, $\tt p^{x'}$ admits
the explicit form
\begin{equation*}
\tt p^{x'}(s,t,x,x')
=\ff{\e^{-\ff{1}{2(t-s)}\<(\si\si^*)^{-1}(x')(x'-x-b(x')(t-s)),x'-x-b(x')(t-s)\>}}{\ss{(2\pi(t-s))^d\mbox{det}((\si\si^*)(x'))}}.
\end{equation*}
A direct calculation yields
\begin{equation}
\pp_s\tt p^{x'}(s,t,x,x')=-\tt L^{x'}\tt
p^{x'}(s,t,x,x'),~~t>s,~~\tt
p^{x'}(s,t,\cdot,x')\to\dd_{x'}(\cdot),~~~s\uparrow t,
\end{equation}
where $\tt L^{x'}$ is the infinitesimal generator of \eqref{A2}.  By
\eqref{A1} and \eqref{A2}, we derive from \cite[(3.8)]{KM} that
\begin{equation}\label{A3}
\begin{split}
p(s,t,x,x')&=\tt
p^{x'}(s,t,x,x')+\int_s^t\int_{\R^d}p(s,u,x,z)H(u,t,z,x')\d z\d u,
\end{split}
\end{equation}
where
\begin{equation}\label{A01}
\begin{split}
 H(s,t,x,&x'):=(L-\tt L^{x'})\tt p^{x'}(s,t,x,x')\\
 &=\<b(x)-b(x'),\nn \tt p^{x'}(s,t,x,x')\>+\ff{1}{2}\<(\si\si^*)(x)-(\si\si^*)(x'),\nn^2
 \tt p^{x'}(s,t,x,x')\>_{\rm HS}.
 \end{split}
\end{equation}
In \eqref{A3}, iterating  for $p(s,u,x,z)$  gives
\begin{equation}\label{A6}
p(s,t,x,x')=\sum_{i=0}^\8(\tt p^{x'}\otimes H^{(i)})(s,t,x,x'),
\end{equation}
where  $\tt p\otimes H^{(0)}:=\tt p$ and  $\tt p^{x'}\otimes
H^{(i)}:=(\tt p^{x'}\otimes H^{(i-1)})\otimes H,i\ge1$, with
\begin{equation*}
(f\otimes g)(s,t,x,x'):=\int_s^t\int_{\R^d}f(s,u,x,z)g(u,t,z,y)\d
u\d z.
\end{equation*}
If we can claim that
\begin{equation}\label{A9}
|\tt p\otimes H^{(i)}|(s,t,x,x')\le
\ff{\e^{\ff{\|b\|_\8^2T}{2\hat\ll_0}}\bb_T^i}{\GG(1+\ff{i}{2})}\,
p_0(t-s,x,x'),
\end{equation}
in which $\bb_t,p_0$ were introduced  in \eqref{A00}, then
\eqref{A0} follows from \eqref{A6} and \eqref{A9}.  Below it
suffices to show that \eqref{A9} holds true. By means of \eqref{S0}
and $|a-b|^2\ge\ff{1}{2}|a|^2-|b|^2,a,b\in\R^d$, it follows from
\eqref{F0} and $\|b\|_\8<\8$ that
\begin{equation}\label{AE1}
\begin{split}
|\nn\tt p|(s,t,x,x')
&\le\ff{
  \,\ss{\hat\ll_0}
\e^{\ff{\|b\|_\8^2T}{4\hat\ll_0}}}{\breve{\ll}_0\ss{t-s}} p_0(t-s,x,x')\\
 \|\nn^2\tt p\|_{\rm
HS}(s,t,x,x')&\le\ff{  (\ss d+\ff{4}{\e})
\e^{\ff{\|b\|_\8^2T}{4\hat\ll_0}}}
{\breve{\ll}_0(t-s)}\ff{\e^{-\ff{|x'-x|^2}{8\hat\ll_0(t-s)}}}
{(2\pi\breve{\ll}_0(t-s))^{d/2}}.
\end{split}
\end{equation}
Thus, combining \eqref{S0} with \eqref{AE1}, besides $\|b\|_\8<\8$
and \eqref{F00}, enables us to obtain
\begin{equation}\label{A4}
|H|(s,t,x,x') \le \ff{ 2\hat\ll_0 \Big\{ \|b\|_\8/\ss{\hat\ll_0}+L_0
 (
d+2\ss d)
\Big\}\e^{\ff{\|b\|_\8^2T}{4\hat\ll_0}}}{\breve{\ll}_0\ss{t-s}}\,
p_0(t-s,x,x').
\end{equation}
 By $
\int_s^t(t-u)^{-\ff{1}{2}}(u-s)^\aa\d
u=(t-s)^{\aa+\ff{1}{2}}B(1+\aa,1/2), t>s, \aa>-1, $ we have
\begin{equation*}\label{A10}
\begin{split}
\Lambda_i(s,t):=\int_s^t\cdots&\int_s^{u_{i-1}}(t-u_1)^{-\ff{1}{2}} \cdots(u_{i-1}-u_i)^{-\ff{1}{2}}
\d u_i\cdots\d u_1=\ff{(\pi(t-s))^{\ff{i}{2}}}{\GG(1+\ff{i}{2})},~~~i\ge1.
\end{split}
\end{equation*}
Whence, taking advantage of  $\|b\|_\8<\8$, \eqref{F0}, \eqref{A4}
as well as
\begin{equation*}
\int_{\R^d}p_0(u-s,x,z) p_0(t-u,y,z)  \d
z=\Big(\ff{8\hat\ll_0}{\breve{\ll}_0}\Big)^d p_0(t-s,x,x'),~~~s<u<t
\end{equation*}
 yields
 \eqref{A9}.
\end{proof}

For  $x,x'\in\R^d$  and $j\ge0, $ let $ (\tt
X^{(\dd),j,x,x'}_{i\dd})_{i\ge j}$ solve the following frozen EM
scheme associated with \eqref{1.1}
\begin{equation*}
\tt X^{(\dd),j,x,x'}_{(i+1)\dd}=\tt
X^{(\dd),j,x,x'}_{i\dd}+b(x')\dd+\si(x')(W_{(i+1)\dd}-W_{i\dd}),~~~i\ge
j,~~~\tt X_{j\dd}^{(\dd),j,x,x'}=x.
\end{equation*}
Write $\tt p^{(\dd),x'}(j\dd,j'\dd,x,y)$ by the transition
 density of $ \tt X^{(\dd),j,x,x'}_{j'\dd}$ at the point $y,$ given $ \tt X^{(\dd),j,x,x'}_{j\dd}=x$.

The following lemma reveals explicit upper bounds of coefficients
with regard to Gaussian bound of the discrete-time EM scheme.
\begin{lem}\label{THA}
Under $\|b\|_\8<\8$ and $({\bf A3})$,  for any $0\le j<j'\le \lfloor
T/\dd\rfloor$
\begin{equation}\label{C1}
p^{(\dd)}(j\dd,j'\dd,x,x')\le\e^{\ff{\|b\|_\8T}{2\hat\ll_0}}\sum_{k=0}^{\8}\ff{\Big(\ss{\pi
T}\hat C_T
((1+24d)\hat\ll_0/\breve{\ll}_0)^d\Big)^k}{\GG(1+\ff{k}{2})}\ff{\e^{-\ff{|
x'-x|^2}{4(1+24d)\hat\ll_0(j'-j)\dd}}}{(2
\pi\breve{\ll}_0(j'-j)\dd)^{d/2}}.
\end{equation}
\end{lem}

\begin{proof}
To obtain \eqref{C1}, we refine the proof of \cite[Lemma 4.1]{LM}.
For $\psi\in C^2(\R^d;\R)$ and $j\ge0$, set
\begin{equation*}
(\mathscr{L}_{j\dd}^{ (\dd) }\psi)(x):=
\dd^{-1}\{\E(\psi(X_{(j+1)\dd}^{(\dd)})|X_{j\dd}^{(\dd)}=x)-\psi(x)
\}, (\hat{\mathscr{L}}_{j\dd}^{ (\dd) }\psi)(x):=\dd^{-1}\{\E
\psi(\tt X^{(\dd),j,x,x'}_{(j+1)\dd}) -\psi(x)\}
\end{equation*}
and
\begin{equation*}
\begin{split}
H^{(\dd)}(j\dd,j'\dd,x,x'):&=(\mathscr{L}_{j\dd}^{ (\dd)
}-\hat{\mathscr{L}}_{j\dd}^{ (\dd) })\tt
p^{(\dd),x'}((j+1)\dd,j'\dd,x,x'),
~j'\ge j+1.
\end{split}
\end{equation*}
In what follows, let $0\le j<j'\le\lfloor T/\dd\rfloor.$ According
to \cite[Lemma 3.6]{KM}, we have
\begin{equation}\label{C6}
p^{(\dd)}(j\dd,j'\dd,x,x')=\sum_{k=0}^{j'-j}(\tt
p^{(\dd),x'}\otimes_\dd H^{(\dd),(k)})(j\dd,j'\dd,x,x'),
\end{equation}
where $\tt p^{(\dd),x'}\otimes_\dd H^{(\dd),(0)})=\tt p^{(\dd),x'},$
$H^{(\dd),(k)} =H^{(\dd)}\otimes_\dd H^{(\dd),(k-1)}$ with
$\otimes_\dd$ being the convolution type binary operation defined by
\begin{equation*}
(f\otimes_\dd
g)(j\dd,j'\dd,x,x')=\dd\sum_{k=j}^{j'-1}\int_{\R^d}f(j\dd,k\dd,x,u)g(k\dd,j'\dd,u,x')\d
u.
\end{equation*}
If the assertion
\begin{equation}\label{A11}
H^{(\dd)}(j\dd,j'\dd,x,x')\le
 \ff{\hat C_T}{\ss{(j'-j)\dd}}\ff{\e^{-\ff{| x'-x|^2}{4(1+24d)\hat\ll_0(j'-j)\dd}}}{(2
\pi\breve{\ll}_0(j'-j)\dd)^{d/2}}
\end{equation}
holds true, where $\hat C_T$ was given in \eqref{C0}, then
\eqref{C1} follows due to \eqref{C6} by an induction argument. So,
in order to complete the proof of Lemma \ref{THA}, it remains to
verify \eqref{A11}.
First of all, we show \eqref{A11} for $j'=j+1.$ By the definition of
$H^{(\dd)}$, observe from \eqref{F0} that
\begin{equation*}
\begin{split}
&|H^{(\dd)}|(j\dd,(j+1)\dd,x,x')=\ff{1}{\dd}|p^{(\dd)}-\tt
p^{(\dd),x'}|(j\dd,(j+1)\dd,x,x') \\
&\le\ff{1}{\dd (2\pi\breve{\ll}_0\dd)^{d/2}
}\bigg\{\Big|\e^{-\ff{1}{2\dd}|(\si\si^*)^{-\ff{1}{2}}(x)(x'-x-b(x)\dd)|^2}
-\e^{-\ff{1}{2\dd}|(\si\si^*)^{-\ff{1}{2}}(x)(x'-x-b(x')\dd)|^2}\Big|\\
&\quad+\Big|
\e^{-\ff{1}{2\dd}\<(\si\si^*)^{-1}(x)(x'-x-b(x')\dd),x'-x-b(x')\dd\>}-\e^{-\ff{1}{2\dd}
\<(\si\si^*)^{-1}(x')(x'-x-b(x')\dd),x'-x-b(x')\dd\>}\Big|\\
&\quad+\ff{1}{2\breve{\ll}_0^d}\e^{-\ff{1}{2\dd}
|(\si\si^*)^{-\ff{1}{2}}(x')(x'-x-b(x')\dd)|^2}
   |\mbox{det}((\si\si^*)(x'))
- \mbox{det}((\si\si^*)(x))|
 \bigg\}\\
&=:\ff{1}{\dd (2\pi\breve{\ll}_0\dd)^{d/2}
}\{\Lambda_1+\Lambda_2+\Lambda_3\}.
\end{split}
\end{equation*}
Next, we aim to estimate $\Lambda_1,\Lambda_2,\Lambda_3$,
one-by-one. By $\|b\|_\8<\8$, \eqref{F0} and \eqref{S0}, it follows
from the first fundamental theorem of calculus that
\begin{equation}\label{C3}
\begin{split}
 |\Lambda_1| 
 \le2\ss{\dd/\breve{\ll}_0}\|b\|_\8\e^{\ff{\|b\|_\8^2\dd}{\hat\ll_0}}
\e^{-\ff{|x-x'|^2}{8\hat\ll_0\dd}}.
\end{split}
\end{equation}
\eqref{F0} and \eqref{F00} imply
\begin{equation*}
\begin{split}
&\|(\si\si^*)^{-1}(x)-(\si\si^*)^{-1}(x')\|_{\rm
HS} 
 \le2\breve{\ll}_0^{-2}\ss{d\hat \ll_0}L_0|x-x'|.
\end{split}
\end{equation*}
This, by invoking  $|\e^a-\e^b|\le\e^{a\vee b} |a-b|$, $a,b\in\R,$
and utilizing  $\|b\|_\8<\8$, \eqref{F0} and \eqref{S0}, yields
\begin{equation}\label{C4}
\begin{split}
|\Lambda_2|
&\le
4\ss{d\dd}L_0(\hat\ll_0/\breve{\ll}_0)^{2}\e^{\ff{\|b\|_\8^2\dd}{4\hat\ll_0}}\e^{-\ff{|x-x'|^2}{16\hat\ll_0\dd}}.
\end{split}
\end{equation}
Also,  making use of $\|b\|_\8<\8$, \eqref{F0} and \eqref{S0}, in
addition to
\begin{equation*}
\begin{split}
|\mbox{det}((\si\si^*)(x))-\mbox{det}((\si\si^*)(x'))|
&\le  2d^{\ff{d}{2}+1}d! \hat\ll_0^{d-\ff{1}{2}} L_0|x-x'|,
\end{split}
\end{equation*}
due to \eqref{F0} and \eqref{F00}, we arrive at
\begin{equation}\label{C5}
\begin{split}
|\Lambda_3|
&\le \ss2d^{\ff{d}{2}+1}d! (\hat\ll_0/\breve{\ll}_0)^{d}
L_0\ss{\dd}\e^{ \ff{\|b\|_\8^2\dd}{2\hat\ll_0}}  \e^{-\ff{
|x'-x|^2}{8\hat\ll_0\dd}}.
\end{split}
\end{equation}
We therefore conclude that \eqref{A11} holds with $j'=j+1$ by taking
\eqref{C3}-\eqref{C5} into account. In the sequel, we are going to
show that \eqref{A11} is still available for $j'>j+1.$
 According
to the notion of $H^{(\dd)}$,
\begin{equation*}
\begin{split}
 &H^{(\dd)}(j\dd,j'\dd,x,x')\\
&=\ff{1}{\dd(2\pi)^{m/2}} \int_{\R^m} \e^{-\ff{|z|^2}{2}} \Big\{\tt
p^{(\dd),x'}((j+1)\dd,j'\dd,x+\GG_z(x),x')-\tt
p^{(\dd),x'}((j+1)\dd,j'\dd,x,x')\Big\}\d
z\\
&\quad-\ff{1}{\dd(2\pi)^{m/2}}
\int_{\R^m}\e^{-\ff{|z|^2}{2}}\Big\{\tt
p^{(\dd),x'}((j+1)\dd,j'\dd,x+\GG_z(x'),x')-\tt
p^{(\dd),x'}((j+1)\dd,j'\dd,x,x')  \Big\}\d z,
\end{split}
\end{equation*}
where $\Gamma_z(x):=b(x)\dd+\ss\dd\si(x)z, x\in\R^d,z\in\R^m.$ By
Taylor's expansion, we further have
\begin{equation*}
\begin{split}
&H^{(\dd)}(j\dd,j'\dd,x,x')\\&=\ff{1}{\dd(2\pi)^{m/2}}
\bigg\{\int_{\R^m} \e^{-\ff{|z|^2}{2}}   \<\nn\tt
p^{(\dd),x'}((j+1)\dd,j'\dd,x,x'),\GG_z(x)-\GG_z(x')\>\d z\\
&\quad+ \int_{\R^m}\e^{-\ff{|z|^2}{2}} \<\nn^2\tt
p^{(\dd),x'}((j+1)\dd,j'\dd,x,x'),(\GG_z\GG^*_z)(x)-(\GG_z\GG^*_z)(x')\>_{\rm
HS}\d z\bigg\} \\
&\quad+\ff{1}{2\dd(2\pi)^{m/2}}\int_{\R^m}\int_0^1(1-\theta)^2\e^{-\ff{|z|^2}{2}}\Big\{\nn_{\GG_z(x)}^3\tt
p^{(\dd),x'}((j+1)\dd,j'\dd,x+\theta\GG_z(x),x')\\
&\quad-\nn_{\GG_z(x')}^3\tt
p^{(\dd),x'}((j+1)\dd,j'\dd,x+\theta\GG_z(x'),x')\Big\} \d\theta\d z\\
&=:\Pi_1+\Pi_2+\Pi_3,
\end{split}
\end{equation*}
where $\nn^i$ means the $i$-th order gradient operator.
Employing
\begin{equation*}
\begin{split}\int_{\R^m}\e^{-\ff{|z|^2}{2}}\mbox{trace}(A\si(x)zz^*\si(x)\d
z&=\int_{\R^m}\e^{-\ff{|z|^2}{2}}z^*\si^*(x)A\si(x)z\d
z=(2\pi)^{m/2}\mbox{trace}(\si^*(x)A\si(x))\end{split}\end{equation*}
for a symmetric $d\times d$-matrix and
$\int_{\R^m}\e^{-\ff{|z|^2}{2}}z\d z={\bf0}$ gives
\begin{equation*}
\begin{split}
\Pi_1+\Pi_2
&=H((j+1)\dd,j'\dd,x,x')+\ff{\dd}{2}\<\nn^2\tt
p^{(\dd),x'}((j+1)\dd,j'\dd,x,x'),(bb^*)(x)-(bb^*)(x')\>_{\rm HS},
\end{split}
\end{equation*}
where $H$ was defined as in \eqref{A01} with  $p^{x'} $  replaced by
$\tt p^{(\dd),x'}$. \eqref{AE1} and  \eqref{A4} enable us to obtain
\begin{equation}\label{L3}
\begin{split}
|\Pi_1|+|\Pi_2|
&\le\ff{2^{\ff{d+1}{2}}\e^{\ff{\|b\|_\8^2T}{4\hat\ll_0}}}{\breve{\ll}_0}
\bigg\{ 2\ss{\hat\ll_0} \|b\|_\8+  (\|b\|_\8^2+2\hat\ll_0L_0\ss
d)(\ss d+2)
 \bigg\}\ff{p_0((j'-j)\dd,x,x')}{\ss{(j'-j)\dd}}.
\end{split}
\end{equation}
Note that $\Pi_3$ can be reformulated as below
\begin{equation*}
\begin{split}
\Pi_3&=\ff{1}{2\dd(2\pi)^{m/2}}\int_{\R^m}\int_0^1(1-\theta)^2\e^{-\ff{|z|^2}{2}}\Big\{\nn_{\GG_z(x)}^3\tt
p^{(\dd),x'}((j+1)\dd,j'\dd,x+\theta\GG_z(x'),x')\\
&\quad-\nn_{\GG_z(x')}^3\tt
p^{(\dd),x'}((j+1)\dd,j'\dd,x+\theta\GG_z(x'),x')\Big\} \d\theta\d
z\\
&\quad+\ff{1}{2\dd(2\pi)^{m/2}}\int_{\R^m}\int_0^1(1-\theta)^2\e^{-\ff{|z|^2}{2}}\Big\{\nn_{\GG_z(x)}^3\tt
p^{(\dd),x'}((j+1)\dd,j'\dd,x+\theta\GG_z(x),x')\\
&\quad-\nn_{\GG_z(x)}^3\tt
p^{(\dd),x'}((j+1)\dd,j'\dd,x+\theta\GG_z(x'),x')\Big\} \d\theta\d
z=:\Pi_{31}+\Pi_{32}.
\end{split}
\end{equation*}
By means of \eqref{F0}, \eqref{F00} and \eqref{S0},  it follows that
\begin{equation}\label{L1}
\begin{split}
|\Pi_{31}|
&\le \ff{2^{m+\ff{d+21}{2}}(L_0+2\|b\|_\8)(\|b\|_\8^2
 +d\hat\ll_0)\Big(1+\ss{2(1+4d)\hat\ll_0}\Big)\e^{\ff{3\|b\|_\8^2T}{8d\hat\ll_0}}}{\breve{\ll}_0^{\ff{3}{2}}((j'-j)\dd)^{\ff{1}{2}}}\\
 &\quad\times\ff{\e^{-\ff{|
x'-x |^2}{8(1+4d)\hat\ll_0(j'-j)\dd}}}{(2
\pi\breve{\ll}_0(j'-j)\dd)^{d/2}},
\end{split}
\end{equation}
Also, by exploiting \eqref{F0}, and \eqref{S0}, we  infer from
Taylor expansion
\begin{equation}\label{L2}
\begin{split}
|\Pi_{32}|
&\le
\ff{2^{m+\ff{d+23}{2}}(L_0+2\|b\|_\8)(\|b\|_\8^3+(d\hat\ll_0)^{\ff{3}{2}})\Big(1+\ss{2(1+24d)\hat\ll_0}\Big)\e^{\ff{(6\|b\|_\8^2
+\|b\|_\8)T}{24d\hat\ll_0}}}{\breve{\ll}_0^2((j'-j)\dd)^{\ff{1}{2}}}
\\
 &\quad\times\ff{\e^{-\ff{| x'-x|^2}{4(1+24d)\hat\ll_0(j'-j)\dd}}}{(2
\pi\breve{\ll}_0(j'-j)\dd)^{d/2}}.
\end{split}
\end{equation}
Consequently, \eqref{A11} follows from \eqref{L3}, \eqref{L1}, and
\eqref{L2}.

\end{proof}

\beg{thebibliography}{99} {\small

\setlength{\baselineskip}{0.14in}
\parskip=0pt

\bibitem{BHY}Bao, J., Huang, X., Yuan, C., Convergence rate of Euler--Maruyama Scheme for SDEs with
H\"older--Dini continuous drifts,  {\it J. Theoret. Probab.},   {\bf
32} (2019),    848--871.

\bibitem{DKS}Dareiotis, K., Kumar, C., Sabanis, S., On
tamed Euler approximations of SDEs driven by L\'{e}vy noise with
applications to delay equations, {\it SIAM J. Numer. Anal.}, {\bf
54} (2016),  1840--1872.

\bibitem{FGP}  Flandoli, M.,   Gubinelli, M.,   Priola,  E.,  Flow of diffeomorphisms
for SDEs with unbounded H?lder continuous drift, {\it Bull. Sci.
Math.}, {\bf 134} (2010),   405--422.

\bibitem{GLN} Gottlich, S., Lux, K.,  Neuenkirch, A.,   The Euler scheme
for stochastic differential equations with discontinuous drift
coefficient: A numerical study of the convergence rate,
arXiv:1705.04562.

\bibitem{GMY}Guo, Q.,  Mao, X.,  Yue, R.,  The truncated
Euler-Maruyama method for stochastic differential delay equations,
{\it Numer. Algorithms}, {\bf 78} (2018),  599--624.

\bibitem{GM} Gy\"ongy, I.,   Martinez, T.,   On stochastic differential equations with locally unbounded drift,  {\it Czechoslovak Math.J.},
  {\bf 51} (2001), 763--783.

\bibitem{GR}Gy\"ongy, I., R\'{a}sonyi, M.,  A note on Euler
approximations for SDEs with H\"older continuous diffusion
coefficients, {\it Stoch. Process. Appl.}, {\bf 121} (2011),
2189--2200.

\bibitem{HK}Halidias, N.,   Kloeden, P.~E., A note on the Euler-Maruyama
scheme for stochastic differential equations with a discontinuous
monotone drift coefficient, {\it BIT},  {\bf48} (2008),  51--59.

\bibitem{HMS} Higham, D.~J., Mao, X.,
  Stuart, A.~M.,
   Strong convergence of Euler-type methods for nonlinear stochastic differential equations, {\it SIAM J. Numer. Anal.}, {\bf 40} (2002),   1041--1063.

\bibitem{HMY}Higham, D.~J., Mao, X., Yuan, C.,
 Almost sure and moment exponential stability in the numerical
simulation of stochastic differential equations, {\it SIAM J. Numer.
Anal.}, {\bf 45}, 592--609.

\bibitem{HW18} Huang,  X.,   Wang, F.-Y.,  Distribution Dependent SDEs with Singular Coefficients,  to appear
in {\it Stoch. Process. Appl.},
 https://doi.org/10.1016/j.spa.2018.12.012.

\bibitem{HJK}Hutzenthaler, M., Jentzen, A.,
Kloeden, P.~E., Strong and weak divergence in finite time of Euler's
method for stochastic differential equations with non-globally
Lipschitz continuous coefficients, {\it Proc. R. Soc. Lond. Ser. A
Math. Phys. Eng. Sci.}, {\bf 467} (2011),   1563--1576.

\bibitem{HJP} Hutzenthaler, M., Jentzen, Arnulf.,
Kloeden, P.~E., Strong convergence of an explicit numerical method
for SDEs with nonglobally Lipschitz continuous coefficients, {\it
Ann. Appl. Probab.}, {\bf 22} (2012),   1611--1641.

\bibitem{JMY} Jentzen, A., Mu\"ller-Gronbach, T., Yaroslavtseva, L., On
stochastic differential equations with arbitrary slow convergence
rates for strong approximation,  {\it Commun. Math. Sci.}, {\bf 14}
(2016), 1477--1500.

\bibitem{KP} Kloeden, P.~E., Platen, E., \emph{Numerical Solution of Stochastic
Differential Equations}, Springer, Berlin, 1992.

\bibitem{KPS} Kloeden, P.~E., Platen, E., Schurz, H., \emph{ Numerical
solution of SDE through computer experiments},  Springer-Verlag,
Berlin, 1994.

\bibitem{KM}Konakov, V., Mammen, E., Local limits theorems for
transition densities of Markov chains converging to diffusions, {\it
Probab. Theory Relat. Fields}, {\bf 117} (2000), 551-587.
\bibitem{KR}   Krylov,  N.~V.,  R\"{o}ckner, M.,  Strong solutions of stochastic equations with singular time dependent drift,  {\it
Probab. Theory Related Fields}, {\bf 131} (2005), 154--196.

\bibitem{LM} Lemaire, V., Menozzi, S., On some  non asymptotic bounds
for the Euler scheme, {\it Electron. J. Probab.}, {\bf 15} (2010),
  1645--1681.

\bibitem{LS}Leobacher, G.,   Sz\"olgyenyi, M.,  A numerical method for SDEs
with discontinuous drift, {\it BIT}, {\bf 56} (2016), 151--162.

\bibitem{LS2} Leobacher, G.,   Sz\"olgyenyi, M.,   A strong order 1/2
method for multidimensional SDEs with discontinuous drift, {\it Ann.
Appl. Probab.}, {\bf 27} (2017), 2383--2418.

\bibitem{LS3}Leobacher, G.,   Sz\"olgyenyi, M. , Convergence of the
Euler-Maruyama method for multidimensional SDEs with discontinuous
drift and degenerate diffusion coefficient, {\it Numer. Math.}, {\bf
138}   (2018), 219--239.

 \bibitem{Mao}Mao, X.,  The truncated Euler-Maruyama method for stochastic differential equations,  {\it J. Comput. Appl. Math.}, {\bf 290} (2015), 370--384.

\bibitem{MY06}Mao, X., Yuan, C., \emph{Stochastic differential equations
with Markovian switching}, Imperial College Press, London, 2006.

\bibitem{MY}M\"uller-Gronbach, T.,   Yaroslavtseva, L.,    On the performance
of the Euler-Maruyama scheme for SDEs with discontinuous drift
coefficient, arXiv:1809.08423.

\bibitem{NSS}Neuenkirch, A., Sz\"olgyenyi, M.,  Szpruch, L.,  An adaptive
Euler-Maruyama scheme for stochastic differential equations with
discontinuous drift and its convergence analysis, {\it SIAM J.
Numer. Anal.}, {\it  57 } (2019), 378--403.

\bibitem{NT}  Ngo, H-L.,  Taguchi, D.,  Strong rate of convergence for the Euler-Maruyama
approximation of stochastic differential equations with irregular
coefficients,  {\it Math. Comp.}, {\bf 85} (2016), 1793--1819.

\bibitem{NT2}
Ngo, H.-L., Taguchi, D.,  On the Euler--Maruyama approximation for
one-dimensional stochastic differential equations with irregular
coefficients, {\it  IMA J. Numer. Anal.}, {\bf 37} (2017),
1864--1883.

\bibitem{PT} Pamen, O.M., Taguchi, D.: Strong rate of convergence for the
Euler--Maruyama approximation of SDEs with H\"older continuous drift
coefficient. arXiv: 1508.07513v1

\bibitem{RZ}R\"ockner, M., Zhang, X., Well-posedness of distribution dependent SDEs with singular
drifts, arXiv:1809.02216.

\bibitem{Sa} Sabanis, S.,  Euler approximations with varying
coefficients: the case of superlinearly growing diffusion
coefficients, {\it  Ann. Appl. Probab.}, {\bf 26} (2016),
2083--2105.

\bibitem{shao} Shao, J., Weak convergence of Euler-Maruyama's approximation for SDEs under
integrability condition, arXiv:1808.07250.

\bibitem{Shi}Shigekawa, I.: Stochastic Analysis, Translations of Mathematical
Monographs, 224, Iwanami Series in Modern Mathematics. American
Mathematical Society, Providence (2004)

\bibitem{XZ}  Xie, L., Zhang,  X.,   Ergodicity of stochastic differential equations with jumps and singular coefficients,    arXiv:1705.07402.

\bibitem{Yan}Yan, L., The Euler scheme with irregular coefficients, {\it Ann.
Probab.}, {\bf 30} (2002),   1172--1194.

\bibitem{Z} Zhang, X.,   Strong solutions of SDEs with singular drift and Sobolev diffusion coefficients,   {\it Stoch. Process. Appl.}, {\bf 115} (2005),
1805--1818.

\bibitem{Z2}Zhang,  X.,   Stochastic homeomorphism flows of SDEs with singular drifts and Sobolev diffusion coefficients,  {\it
Electron. J. Probab.}, {\bf 16} (2011), 1096--1116.

\bibitem{AZ} Zvonkin, A.~K.,   A transformation of the phase space of a diffusion process that removes the drift,   {\it Math. Sb.},  {\bf 93} (1974), 129-149.

}
\end{thebibliography}

\end{document}